\documentclass[12pt,a4paper]{article}
\usepackage{latexsym, amsmath, amsfonts, amsthm, amssymb}
\usepackage{times}
\usepackage{a4wide}
\usepackage{graphicx, tocloft}
\usepackage{mathrsfs}
\usepackage[hidelinks]{hyperref}

\def \tr {\text{\rm Tr}}

\def \Vol {\text{\rm Vol}}

\def \RR {\mathbb R}

\def \EE {\mathbb E}

\def \CC {\mathbb C}
\def \PP {\mathbb P}

\def \eps {\varepsilon}

\def \cT {\mathcal T}

\def \cE {\mathcal E}

\def \cA {\mathcal A}
\def \cS {\mathcal S}

\newtheorem{theorem}{Theorem}[section]

\newtheorem{lemma}[theorem]{Lemma}

\newtheorem{proposition}[theorem]{Proposition}
\newtheorem{corollary}[theorem]{Corollary}

\theoremstyle{definition}
\newtheorem{remark}[theorem]{Remark}

\def\myffrac#1#2 in #3{\raise 2.6pt\hbox{$#3 #1$}\mkern-1.5mu\raise 0.8pt\hbox{$
		#3/$}\mkern-1.1mu\lower 1.5pt\hbox{$#3 #2$}}

\def\qed{\hfill $\vcenter{\hrule height .3mm
		\hbox {\vrule width .3mm height 2.1mm \kern 2mm \vrule width .3mm
			height 2.1mm} \hrule height .3mm}$ \bigskip}

\def \id {{\rm Id}}
\def \diam {{\rm diam}}

\def \div {{ \rm div}}

\def \cov {{ \rm Cov}\,}
\def \var {{ \rm Var}\,}

\DeclareMathOperator{\drift}{drift}
\hypersetup{pdftitle={Presenting a proof of the \\ Kannan-Lov\'asz-Simonovits conjecture},
pdfauthor={Pierre Bizeul, Boaz Klartag, Joseph Lehec}}
\title{Presenting a proof of the \\ Kannan-Lov\'asz-Simonovits conjecture}
\author{Pierre Bizeul \and Boaz Klartag \and Joseph Lehec}
\date{}

\begin{document}

\maketitle

\begin{abstract}
Together with Bourgain's slicing problem and the variance conjecture,
the Kannan-Lov\'asz-Simonovits (KLS) conjecture has been a driving force
for mathematical developments in high dimensional geometry in recent decades.
We present here a proof of this conjecture. The crucial step in the proof
is the very recent criterion from Song and Zhang involving high-order derivatives of tilt averages.
\end{abstract}

\section{Introduction}
The Kannan-Lov\'asz-Simonovits (KLS) conjecture from \cite{KLS} poses little challenge for
contemporary AI tools, given  the remarkable progress in our
understanding of high-dimensional convex bodies in recent decades.

\medskip
In this paper we provide the final step in the proof of this conjecture, which
predicts dimension-free bounds for concentration, the Poincar\'e
inequality and isoperimetry in convex bodies. In its original geometric
formulation, due to Kannan, Lov\'asz and Simonovits \cite{KLS}, the
conjecture asserts that the isoperimetric inequality in a convex body
is saturated, up to a universal constant, by half-spaces.

\medskip
Let $K\subseteq\RR^n$ be a bounded, open, convex set, and let $X$ be
a random vector distributed uniformly in $K$. We say that $K$ is
{\it isotropic} if
\begin{equation}
 \EE X=0,\qquad \EE X_iX_j=\delta_{ij}
 \qquad (i,j=1,\ldots,n).
\label{eq_1201} \end{equation}
Any convex body with a non-empty interior may be placed in isotropic
position by an invertible affine transformation. In this normalization,
the KLS conjecture has the following three equivalent formulations,
describing the same phenomenon in probabilistic, analytic or geometric terms:

\begin{enumerate}

\item {\it Concentration.}
For any $1$-Lipschitz function $f:K\to\RR$ with $\EE f(X)=0$,
\begin{equation}
 \PP\big(|f(X)|\geq t\big)\leq 2e^{-ct}
 \qquad\text{for all }t>0.
 \label{eq_KLS_concentration}
\end{equation}

\item {\it Poincar\'e inequality.}
For any smooth function $f:K\to\RR$ with $\int_K f=0$,
\begin{equation}
 \int_K f^2\,dx\leq C\int_K|\nabla f|^2\,dx.
 \label{eq_KLS_poincare}
\end{equation}

\item {\it Isoperimetry.}
For any open set $A\subseteq\RR^n$ with a smooth boundary,
\begin{equation}
 \Vol_{n-1}(\partial A\cap K)
 \geq c\min\big\{\Vol_n(A\cap K),\Vol_n(K\setminus A)\big\}.
 \label{eq_KLS_isoperimetry}
\end{equation}

\end{enumerate}

\medskip
Here, $c,C>0$ are universal constants, independent of the dimension
and of the convex body. The equivalence of these formulations,
possibly with different universal constants, follows from the works of
Cheeger \cite{Cheeger}, Buser \cite{Buser}, Ledoux \cite{Ledoux},
Gromov and V.~Milman \cite{GM}, and E.~Milman~\cite{EM}.

\medskip
We prove the conjecture in the more general setting of log-concave
probability measures. A probability measure $\mu$ on $\RR^n$ with
a density is called {\it log-concave} if its density has the form
$e^{-V}$, where $V:\RR^n\to\RR\cup\{+\infty\}$ is convex.
This class includes uniform measures on convex bodies as well as
Gaussian measures. The probability measure $\mu$ is isotropic if its barycenter
is at the origin and its covariance matrix is the identity; see (\ref{eq_1201}) above.
The {\it Poincar\'e constant} $C_P(\mu)$ is the smallest constant
for which
\begin{equation}
 \var_\mu(f)\leq C_P(\mu)\int|\nabla f|^2\,d\mu
 \label{eq_poincare}
\end{equation}
for all locally-Lipschitz functions $f\in L^2(\mu)$, where
$
 \var_\mu(f)=\int f^2\,d\mu-\left(\int f\,d\mu\right)^2.
$
Our main result is the following:

\begin{theorem}
\label{thm_main}
There exists a universal constant $C>0$ such that, for any dimension
$n\geq1$ and any isotropic log-concave probability measure $\mu$
on $\RR^n$,
$$
 C_P(\mu)\leq C.
$$
\end{theorem}

\medskip
Equivalently, for any log-concave probability measure
$\mu$ on $\RR^n$,
\begin{equation}
 \|\cov(\mu)\|_{op}
 \leq C_P(\mu)
 \leq C\|\cov(\mu)\|_{op},
 \label{eq_KLS_covariance}
\end{equation}
where $\cov(\mu) \in \RR^{n \times n}$ is the covariance matrix of $\mu$ and $\| \cdot \|_{op}$ is its operator norm.
The left-hand inequality in (\ref{eq_KLS_covariance}) follows by testing \eqref{eq_poincare}
on linear functions. The right-hand inequality follows from
Theorem~\ref{thm_main} by an affine change of variables.
It follows that up to a universal factor, the Poincar\'e constant
of a log-concave measure is determined by testing linear functions.

\medskip
The study of Poincar\'e inequalities on convex domains goes back to
Poincar\'e \cite{Poincare} in the 1890s.
The sharp diameter bound in the Poincar\'e inequality
from Payne and Weinberger \cite{PW} is
$$
 C_P(\mu_K)\leq \frac{\diam(K)^2}{\pi^2},
$$
where $\mu_K$ is the uniform probability measure on a convex body $K$.
The KLS conjecture, formulated in 1995, replaces the diameter by a
quantity that describes the distribution of mass rather than the
location of the most extreme points: the largest eigenvalue of the
covariance matrix. This distinction is already significant for an
isotropic Euclidean ball or cube, whose squared diameter is of order
$n$, and for an isotropic regular simplex, whose squared diameter is
of order $n^2$. The conjecture arose in the study of localization and
random walks in convex bodies, with applications to random sampling
and volume computation \cite{KLS}.

\medskip
Two closely related problems concern hyperplane sections and the thin shell phenomenon. Bourgain's slicing problem, originating in his work on
maximal functions \cite{Bourgain} in the 1980s, asks whether every
convex body of volume one in $\RR^n$ has a hyperplane section of
$(n-1)$-dimensional volume at least a positive universal constant.
The thin-shell problem emerged from the study of the central limit
phenomenon for convex bodies by Anttila, Ball and Perissinaki
\cite{ABP}, and was formulated as a variance conjecture by Bobkov
and Koldobsky \cite{BK}. It asks whether every isotropic log-concave
random vector $X$ in $\RR^n$ satisfies
\begin{equation}
 \var(|X|^2)\leq Cn.
 \label{eq_thin_shell_history}
\end{equation}
Such an estimate implies that
$\EE(|X|-\sqrt{n})^2\leq C$, and hence that most of the mass lies in
a spherical shell of bounded width around the sphere of radius
$\sqrt{n}$.
The KLS conjecture implies \eqref{eq_thin_shell_history} simply by
applying the Poincar\'e inequality to $x\mapsto |x|^2$:
$$
 \var(|X|^2)\leq 4C_P(\mu)\EE|X|^2=4nC_P(\mu),
$$
where $\mu$ is the law of $X$. Eldan and Klartag \cite{EK} proved
that the thin-shell conjecture implies the slicing conjecture,
while Ball and Nguyen \cite{BN} provided a more direct implication from KLS to slicing.
See also Bizeul \cite{Biz}.
Thus the three problems formed a hierarchy, with KLS at the strongest
end. A partial converse was obtained by Eldan \cite{eldan}, who
introduced stochastic localization and showed that a thin-shell bound
yields a KLS bound with a loss of polylogarithmic factors in the
dimension.

\medskip
Stochastic localization became a principal tool for
improving the dimension dependence in these inequalities. Lee and
Vempala \cite{LV} obtained $C_P(\mu)\leq C\sqrt{n}$ for isotropic
log-concave measures, and a breakthrough result by Chen \cite{Chen} established a sub-polynomial
bound. Klartag and Lehec \cite{KL} proved a
poly-logarithmic bound. This was followed by
\begin{equation}
 C_P(\mu)\leq C\log n,
 \label{eq_logarithmic_KLS_history}
\end{equation}
proved by Klartag \cite{KlartagLog} using an improved log-concave
Lichnerowicz inequality together with stochastic localization.
These developments also yielded  progressively stronger estimates for
the slicing and thin-shell problems.

\medskip
The slicing problem was resolved by Klartag and Lehec
\cite{KLslicing}, building on a bound of Guan \cite{Guan}.
Their proof combines stochastic localization, Milman's theory of
$M$-ellipsoids and stability estimates for the Shannon-Stam
inequality. Bizeul \cite{BizeulSlicing} gave an alternative proof
through small-ball estimates, again using stochastic localization, $M$-ellipsoids
and Guan's bound. A simpler proof without stochastic localization was
recently found by Brazitikos \cite{bra}. Klartag and Lehec \cite{KLthin} subsequently
resolved the thin-shell conjecture by constructing a parallel
coupling between exponential tilts of a log-concave probability measure.

\medskip
A different proof of the thin-shell inequality, and consequently of the slicing conjecture, was obtained in
Chen and Klartag \cite{ChenKlartag}. It yields the sharp estimate
\begin{equation}
 \var(|X|^2)\leq 8n
 \label{eq_sharp_thin_shell_history}
\end{equation}
for every isotropic, log-concave random vector $X$ in $\RR^n$.
Equality is attained when the coordinates of $X$ are independent
standard centered exponential random variables. Their argument uses
log-concave moment measures from Klartag \cite{lc_moment} and the associated weighted Riemannian
manifold, with its metric given by the Hessian of a convex potential
satisfying a Monge-Amp\`ere equation.
An anisotropic modification of the Chen-Klartag argument
from Bizeul \cite[Section~2]{BizeulMM} and Letwin \cite{letwin}
gives
$$
 \var(|Y|^2)\leq 8\tr(\Sigma^2)
$$
for a centered log-concave random vector $Y$ with covariance matrix
$\Sigma$. It yields a dimension-free bound for the third-moment
parameter appearing in stochastic localization, see (\ref{eq_intro_third}) below. Combining this
bound with the improved log-concave Lichnerowicz inequality from
\cite{KlartagLog} gives
\begin{equation}
 C_P(\mu)\leq C\sqrt{\log n}
 \label{eq_sqrt_log_KLS_history}
\end{equation}
for isotropic log-concave measures. Mikulincer and Zadik \cite{MZ} have recently proved
the KLS conjecture in the  case of unconditional convex bodies.
A very recent work by Song and Zhang  \cite{SZ},
which we discuss in great detail below, introduces several new ideas and improves
the bound on the right-hand side of (\ref{eq_sqrt_log_KLS_history}) to $C \cdot 16^{\log^* n}$.
Theorem~\ref{thm_main} removes the remaining  dependence
on the dimension.

\medskip Our starting point for the proof of Theorem~\ref{thm_main}
is the aforementioned bound on the third centered moment:
for every isotropic log-concave random vector $X$ and every $u\in\RR^n$,
\begin{equation}
 \left|\EE\big[\langle X,u\rangle X\otimes X\big]\right|^2
 \leq 8|u|^2,
 \label{eq_intro_third}
\end{equation}
where the norm on the left is the Hilbert--Schmidt norm; see
Lemma~\ref{lem_third_cumulant} below. We inductively propagate
this estimate to cumulants of all orders. If
$\Lambda_\mu(z)=\log\int e^{\langle z,x\rangle}\,d\mu(x)$ and
$\kappa_m^\mu=\nabla^m\Lambda_\mu(0)$, the resulting bound is
$$
 |\kappa_m^\mu(u,\cdot,\ldots,\cdot)|^2
 \leq K^{m-1}((m-1)!)^2|u|^2,
 \qquad m\geq2,
$$
with a universal constant $K\geq1$.
The induction step  uses Eldan's stochastic localization \cite{eldan}, with
inverse covariance as the quadratic-variation matrix.

\medskip
To use these estimates, for $f \in L^2(\mu)$ we consider the tilt-average
$$
 F_f(z)=\frac{\int f(x)e^{\langle z,x\rangle}\,d\mu(x)}
                 {\int e^{\langle z,x\rangle}\,d\mu(x)}.
$$
A suspension construction, applied to independent copies of $\mu$,
realizes the derivatives of $F_f$ as mixed cumulants of an isotropic
log-concave measure in a larger dimension. This yields the bound
$$
 \left|\frac{1}{d!}\nabla^dF_f(0)\right|
 \leq R^d\|f\|_{L^2(\mu)}\qquad(d\geq1),
$$
with $R$ bounded from above by a universal constant. A criterion of Song and Zhang
\cite[Theorem~5.1]{SZ} then gives the required Poincar\'e bound.

\medskip The remainder of this paper is organized as follows.
Section~\ref{sec_background} contains the analytic preliminaries, and
Section~\ref{sec_taylor} presents the proof of the Song-Zhang criterion.
The cumulant estimates are proved in Sections~\ref{sec_cumulants}--\ref{sec_induction}.
Section~\ref{sec_suspension} contains the suspension argument and
completes the proof of Theorem~\ref{thm_main}.

\medskip {\it Acknowledgements.} Most proofs and mathematical ideas in this paper were found by ChatGPT; 
a notable exception is the idea to use suspension which was suggested by the authors.
The role of the authors has been mostly to understand these proofs and improve their exposition.
The first two authors were supported by a grant from the Israel Science Foundation (ISF~$\# 674/24$).

\section{The Laplace operator of a log-concave measure}
\label{sec_background}

This section provides standard background on the Laplace operator associated with a log-concave probability measure
in $\RR^n$.
We say that a log-concave probability measure $\mu$ on $\RR^n$ is
{\it regular} if its density is $e^{-V}$, where $V:\RR^n\to\RR$ is smooth,
any partial derivative of $V$ of any order grows at most polynomially at infinity, and there exist $a, b> 0$
such that
\begin{equation}
 a\,\id\leq\nabla^2V(x) \leq b\,\id \qquad \qquad \qquad \text{for all} \ x \in \RR^n
 \label{eq_2256}
\end{equation}
in the sense of symmetric matrices, where $\nabla^2 V(x) \in \RR^{n \times n}$ is the Hessian matrix of $V$ at the point $x$. The parameters $a, b > 0$ may depend on $\mu$, and in particular they need not be universal constants.

\medskip In proving Theorem \ref{thm_main} we may assume that $\mu$ is regular and log-concave,
since any full-dimensional log-concave probability measure may be approximated
by a regular one. Indeed, first convolve with a Gaussian of covariance
$\eps\,\id$, and then multiply the resulting density by
$e^{-\delta|x|^2/2}$ and normalize. The new potential has Hessian between
$\delta\,\id$ and $(\delta+\eps^{-1})\,\id$, and its derivatives have polynomial
growth by Klartag and Putterman \cite[Lemma~2.2(i)]{KP}; see also
\cite[Section~3]{KP}. Letting $\eps,\delta\to0$ gives weak convergence and
convergence of moments. In the isotropic case we may also recenter and
normalize the covariance. Thus uniform Poincar\'e bounds for regular
measures pass to the limit, by testing on smooth, compactly supported
functions. See e.g. \cite{KP} and references therein for more details.

\medskip
Let $\cA$ be the algebra of smooth functions $u:\RR^n\to\RR$ such that,
for every multi-index $\alpha$, there exist $C_\alpha,N_\alpha>0$ with
\begin{equation}
 |\partial^\alpha u(x)|\leq C_\alpha(1+|x|)^{N_\alpha}
 \qquad (x\in\RR^n).
 \label{eq_algebra}
\end{equation}
This algebra contains the polynomials as well as the function $V$, and it is closed under
partial differentiation. Its elements and all their derivatives belong
to $L^p(\mu)$ for every $1\leq p<\infty$. They also have subexponential
decay relative to $e^{-V}$ in the sense of \cite[Section~2]{KP}.
We use this smaller class here because it is closed under products.
The Laplace operator associated with $\mu$ is
\begin{equation}
 Lu=\Delta u-\langle\nabla V,\nabla u\rangle
    =e^V\div(e^{-V}\nabla u).
 \label{eq_L}
\end{equation}
It maps $\cA$ into itself. For $u,v\in\cA$ we have
\begin{align}
 \int_{\RR^n} (Lu)v\,d\mu
 &=-\int_{\RR^n} \langle\nabla u,\nabla v\rangle\,d\mu
 =\int_{\RR^n} u(Lv)\,d\mu.
 \label{eq_ibp}
\end{align}
In particular, $\int Lu\,d\mu=0$. To justify these identities, insert a
smooth cutoff $\chi_R(x)=\chi(x/R)$, equal to one on the ball of radius
$R$ and zero outside the ball of radius $2R$, and integrate by parts.
This way we obtain an error in \eqref{eq_ibp} which is
$-\int v\langle\nabla\chi_R,\nabla u\rangle\,d\mu$.
The error tends to zero as $R \to \infty$ by our polynomial growth assumptions and the Gaussian decay.  The {\it Dirichlet energy} of a function $f\in\cA$ is the quantity
$$
 \|\nabla f\|_{L^2(\mu)}^2
 =\int_{\RR^n}|\nabla f|^2\,d\mu
 =-\int_{\RR^n}fLf\,d\mu.
$$
Differentiating \eqref{eq_L} gives
$\partial_iLu=L\partial_i u-\sum_j(\partial_{ij}V)\partial_j u$.
It is well-known that applying \eqref{eq_ibp} twice yields the {integrated Bochner
identity},
\begin{equation}
 \int(Lu)^2\,d\mu
 =\int\|\nabla^2u\|_{HS}^2\,d\mu
  +\int\langle\nabla^2V\nabla u,\nabla u\rangle\,d\mu
 \qquad (u\in\cA).
 \label{eq_bochner}
\end{equation}
Initially defined on $C_c^\infty(\RR^n)$, the operator $L$ is essentially
self-adjoint in $L^2(\mu)$. We use the same notation $L$ for its self-adjoint
closure. The operator $-L$ has discrete spectrum
$0=\lambda_0<\lambda_1\leq\lambda_2\leq\cdots\to\infty$, with
$\lambda_1=1/C_P(\mu)$; see \cite[Section~1 and Appendix~A]{KP}.
Indeed, the isometry $u\mapsto e^{-V/2}u$ transforms $-L$ to
the Schr\"odinger operator \begin{equation}
 -\Delta+W,\qquad
 W=\frac{|\nabla V|^2}{4}-\frac{\Delta V}{2},
 \qquad W(x)\longrightarrow\infty\quad\text{as }|x|\to\infty.
 \label{eq_schrodinger}
\end{equation}
The last assertion follows from \eqref{eq_2256}.

\medskip
Every eigenfunction of $L$ belongs to $\cA$. Here is a sketch of a justification of
this fact which also follows from Wang's Harnack inequality, see e.g. \cite[Theorem 5.6.1]{BGL}.  Let $Lf=-\lambda f$. Elliptic regularity makes
$f$ smooth, and it is explained in
\cite[Remark~2.7]{KP} that  $\psi=e^{-V/2}f\to0$ at infinity.
Choose an integer $m$ with $2ma>\lambda$ and put
$w(x)=(1+|x|^2)^m$. Differentiating, we see that
\begin{equation}
 \langle x,\nabla V(x)\rangle
 \geq a|x|^2- O( |x|),
 \qquad
 \frac{Lw}{w}\leq-2ma+o(1)\qquad \qquad \qquad \text{as }|x|\to\infty.
 \label{eq_1338} \end{equation}
Thus, outside a sufficiently large ball, the function $W$ from (\ref{eq_schrodinger})
satisfies $W>\lambda$ and
$(-\Delta+W-\lambda)(e^{-V/2}w)\geq0$. Choose $C > 0$ so that
$|\psi|\leq Ce^{-V/2}w$ on the boundary of the ball. The maximum principle, applied separately to
$\psi-Ce^{-V/2}w$ and $-\psi-Ce^{-V/2}w$, gives $|f|\leq Cw$ outside
the ball. The application of the maximum principle is legitimate since both $\psi$ and $e^{-V/2}w$ tend to zero at infinity.
Thus $f$ grows at most polynomially at infinity.

\medskip Interior gradient estimates on the balls
$B(x,(1+|x|)^{-1})$, rescaled to the unit ball, show that
if $u$ and $(L+\lambda)u$ have at most polynomial growth,
then so does $\nabla u$.
Differentiating $Lf=-\lambda f$ and using the polynomial growth
of all derivatives of $V$, we conclude by induction that every
partial derivative of $f$ has at most polynomial growth.
Thus $f\in\cA$.

\medskip A first, non-trivial, normalized eigenfunction of $L$ is therefore
a function  $f\in\cA$ with
\begin{equation}
 Lf=-\lambda_1f,\qquad \int_{\RR^n} f\,d\mu=0,\qquad \int_{\RR^n} f^2\,d\mu=1.
\label{eq_1513} \end{equation}
From \eqref{eq_ibp},
$$
 \int_{\RR^n} |\nabla f|^2\,d\mu=\lambda_1.
$$
Write $\|u\|_2$ for the norm in $L^2(\mu)$, let
$\lambda=\lambda_1=1/C_P(\mu)>0$, and set
$$
 \cA_0=\left\{g\in\cA \, ; \, \int_{\RR^n}g\,d\mu=0\right\}.
$$
The inverse of $-L$ is a well-defined operator on $\cA_0$, as explained in the following:

\begin{lemma}
\label{lem_poisson}
For every $g\in\cA_0$ there exists a unique $u\in\cA_0$ such that
$-Lu=g$. Writing $u=(-L)^{-1}g$, we have
\begin{equation}
 \|u\|_2\leq\lambda^{-1}\|g\|_2,\qquad
 \int_{\RR^n}|\nabla u|^2\,d\mu
 =\int_{\RR^n}gu\,d\mu
 \leq\lambda^{-1}\|g\|_2^2.
 \label{eq_poisson_energy}
\end{equation}
\end{lemma}

\begin{proof}
The spectral theorem gives a unique mean-zero solution
$u\in\operatorname{Dom}(L)$ of $-Lu=g$, with
$\|u\|_2\leq\lambda^{-1}\|g\|_2$.
Local elliptic regularity gives $u \in C^{\infty}(\RR^n)$
with $-Lu = g$ pointwise in $\RR^n$. Integrating by parts, we see that for any compactly-supported, Lipschitz function $h: \RR^n \rightarrow \RR$,
\begin{equation}
 \int_{\RR^n}\langle\nabla u,\nabla h\rangle\,d\mu
 =\int_{\RR^n}gh\,d\mu.
 \label{eq_poisson_weak}
\end{equation}
Let us prove that $u\in\cA$.  For a sufficiently large integer $m$, put
$w(x)=(1+|x|^2)^m$ and let us use~(\ref{eq_1338}).
Since $g$ has polynomial growth, we may choose $m$ and then $R$
so that $-Lw\geq |g|$ outside the ball of radius $R$.
Choose $C\geq1$ such that $|u|\leq Cw$ on this ball.
The function $h=(u-Cw)_+$ is locally Lipschitz, it vanishes on
this ball, and it satisfies $0\leq h\leq |u|$.
Let $\chi_T\in C_c^\infty(\RR^n)$ satisfy $0\leq\chi_T\leq1$,
$\chi_T=1$ on $B(0,T)$, $\chi_T=0$ outside $B(0,2T)$,
and $|\nabla\chi_T|\leq2/T$.
Using $\chi_T^2h$ as a test function in
\eqref{eq_poisson_weak}, and integrating by parts, gives
$$
 \int_{\RR^n}\chi_T^2|\nabla h|^2\,d\mu
 +2\int_{\RR^n}\chi_T h
       \langle\nabla h,\nabla\chi_T\rangle\,d\mu
 =\int_{\RR^n}(g+CLw)\chi_T^2h\,d\mu\leq0.
$$
By the Cauchy--Schwarz inequality,
$$
 \int_{\RR^n}\chi_T^2|\nabla h|^2\,d\mu
 \leq4\int_{\RR^n}h^2|\nabla\chi_T|^2\,d\mu
 \leq\frac{16}{T^2}\|u\|_2^2.
$$
Letting $T\to\infty$ gives $\nabla h=0$ almost everywhere.
Since $h$ vanishes on the ball, $h \equiv 0$.
Applying the same argument to $-u$ shows that
$|u|\leq Cw$ throughout $\RR^n$.

\medskip
Interior gradient estimates on balls of radius
$(1+|x|)^{-1}$ now give polynomial growth of~$\nabla u$.
Differentiating $-Lu=g$ and repeating this argument gives
polynomial growth of every partial derivative of any order of $u$.
Consequently $u\in\cA_0$. The integration by parts \eqref{eq_ibp} gives
$$
 \int_{\RR^n}|\nabla u|^2\,d\mu
 =-\int_{\RR^n}uLu\,d\mu
 =\int_{\RR^n}gu\,d\mu
 \leq\|g\|_2\|u\|_2
 \leq\lambda^{-1}\|g\|_2^2.
$$
This proves \eqref{eq_poisson_energy}.
Uniqueness in $\cA_0$ follows from \eqref{eq_ibp}:
if $v\in\cA_0$ satisfies $Lv=0$, then
$\int|\nabla v|^2\,d\mu=0$, so $v$ is constant and hence zero.
\end{proof}

Let
$$ A \in \RR^{\overbrace{n\times\cdots\times n}^{k\text{ times}}} \cong (\RR^n)^{\otimes k}, $$
i.e., $A$ is a tensor of real numbers with $k$ indices. We represent
$A$ as $A=(A_{i_1,\ldots,i_k})_{i_1,\ldots,i_k =1 ,\ldots,n}$, and write
$$ |A| = \sqrt{ \sum_{i_1,\ldots,i_k=1}^n  |A_{i_1,\ldots,i_k}|^2 } $$
for its Hilbert-Schmidt norm. The associated scalar product of $A, B \in (\RR^n)^{\otimes k}$
is denoted by $\langle A, B \rangle$. For a tensor-valued function $u=(u_{i_1,\ldots,i_k})$ on $\RR^n$,
the expressions $Lu$, $L^{-1}u$ and $\int u\,d\mu$
are understood componentwise.
We write that $u \in \cA$ or $u \in \cA_0$ if each component does.
We use the norm
$$
 \|u\|_2^2 = \|u\|_{L^2(\mu)}^2
 =\sum_{i_1,\ldots,i_k=1}^n
   \int_{\RR^n}|u_{i_1,\ldots,i_k}(x)|^2\,d\mu(x).
$$
Note that our tensors are not necessarily symmetric.
In our notation, differentiating a tensor-valued function adds a new first index. Thus
$$
 \nabla u(x)
 =\big(\partial_j u_{i_1,\ldots,i_k}(x)\big)_{j,i_1,\ldots,i_k=1,\ldots,n}
 \in \RR^{\overbrace{n\times\cdots\times n}^{k+1\text{ times}}} \cong (\RR^n)^{\otimes (k+1)}.
$$
The {\it centered gradient} is denoted by
\begin{equation}
 \nabla_0 u=\nabla u-\int_{\RR^n}\nabla u\,d\mu.
\label{eq_1045} \end{equation}
Since
\begin{equation}
 \|\nabla_0 u\|_2^2
 =\|\nabla u\|_2^2
  -\left|\int_{\RR^n}\nabla u\,d\mu\right|^2,
 \label{eq_centered_gradient}
\end{equation}
we conclude from  \eqref{eq_poisson_energy} that for any $g \in \cA_0$,
\begin{equation}
 \|\nabla_0(-L)^{-1}g\|_2
 \leq\lambda^{-1/2}\|g\|_2.
 \label{eq_inverse_gradient}
\end{equation}

\medskip For a tensor $T$ with at least $q$ indices, denote by $\cS_q T$ its
average over all permutations of the first $q$ indices. We will need  the following lemma:

\begin{lemma}
\label{lem_tensor_recovery}
Let $d \geq 1$. Let $T$ be a tensor with at least $3d$ indices. Assume that
$T$ is invariant under permutations of its first $d$ indices, and also under
permutations of its next $2d$ indices. Then,
$$
 |T|
 \leq4^d|\cS_{2d}T|.
$$
\end{lemma}
\begin{proof}
All sets below are subsets of $\Omega = \{1,\ldots,3d\}$.
For a subset $F$ of cardinality $d$, let $T_F$ be the tensor obtained from $T$ by permuting
its indices in the following manner: we move the first $d$
indices of $T$ to $F$, and the next $2d$ indices to $\Omega \setminus F$.
We leave indices beyond $3d$ untouched.
We claim that
\begin{equation}
 \sum_{|E|=2d}
 \left|\sum_{\substack{F\subseteq E\\|F|=d}}T_F\right|^2
 =
 \sum_{s=0}^d\binom{d}{s}
 \sum_{|S|=s}
 \left|\sum_{\substack{F\supseteq S\\|F|=d}}T_F\right|^2\geq\sum_{|F|=d}|T_F|^2. \label{eq_1412}
\end{equation}
Indeed, expanding the squared norms, the coefficient of
$\langle T_F,T_{F'}\rangle$ in both expressions is
$$
 \binom{3d - (2d-r)}{r} = \binom{d+r}{d}
 =\sum_{s=0}^d\binom{r}{s} \binom{d}{d-s},
 \qquad r=|F\cap F'|.
$$
This proves the equality in (\ref{eq_1412}).
The inequality follows by keeping  only the term $s=d$, and discarding the
other nonnegative terms. Each $T_F$ is an indices-permuted version of $T$ and hence $|T_F| = |T|$. Each inner sum in the expression
on the left-hand side of (\ref{eq_1412}) is an indices-permuted version of $\binom{2d}{d}\cS_{2d}T$.
Hence, from (\ref{eq_1412}),
$$
\binom{3d}{d}|T|^2\leq\binom{3d}{2d}\binom{2d}{d}^2|\cS_{2d}T|^2, $$
and the conclusion follows as $\binom{2d}{d} \leq 4^d$.
\end{proof}

\section{High-order derivatives of the tilt-average}\label{sec_taylor}

In this section we revisit the proof of Theorem~5.1 of Song and Zhang \cite{SZ}.
The idea of using exponential tilts in the context of the slicing problem
goes back to Klartag \cite{K_quarter}. Given a regular,
log-concave probability measure $\mu$ on $\RR^n$, we consider its family of
associated {\it exponential tilts} $(\mu_z)_{z\in\RR^n}$, defined via
\begin{equation}
 \frac{d\mu_z}{d\mu}(x)=e^{x\cdot z-\Lambda_\mu(z)}.
\label{eq_1314} \end{equation}
Here,
$$
 \Lambda_\mu(z)=\log\int_{\RR^n}e^{x\cdot z}\,d\mu(x)
$$
is a normalizing factor, ensuring that $\mu_z$ is a probability measure for any $z \in \RR^n$.
Given $f\in L^2(\mu)$, we consider its averages with respect to the
 tilted measures:
\begin{equation}
 F_f(z)=\int_{\RR^n}f(x)\,d\mu_z(x)
 \qquad (z\in\RR^n). \label{eq_1315}
\end{equation}
The function $F_f$ is smooth, and  its Taylor coefficients at the origin
are given, for $d\geq1$, by
\begin{equation}
 \mathcal{T}_d f
 =\frac{1}{d!}\nabla^d F_f(0)
 =\frac{1}{d!}
 \left(
 \frac{\partial^d F_f}{\partial z_{i_1}\cdots\partial z_{i_d}}(0)
 \right)_{i_1,\ldots,i_d=1,\ldots,n}
 \in\RR^{\overbrace{n\times\cdots\times n}^{d\text{ times}}}.
\label{eq_1337} \end{equation}
The following is a consequence of Theorem 5.1 from Song and Zhang \cite{SZ}:

\begin{theorem} Let $\mu$ be a centered, regular, log-concave probability measure in $\RR^n$ whose covariance matrix
is at most the identity. Assume that $R \geq 1$ is such that for any $f \in L^2(\mu)$ and $d \geq 1$,
\begin{equation}  | \cT_d f | \leq R^d \| f \|_2. \label{eq_1736} \end{equation}
Then $C_P(\mu) \leq C R^2$ where $C > 0$ is a universal constant.
\label{thm_1032}
\end{theorem}

The remainder of this section is devoted to an exposition of a proof of Theorem \ref{thm_1032}, along the lines of \cite{SZ}.
One of the proof ideas is reminiscent of the
 higher-order Stein-kernel
construction of Fathi \cite{fathi}; integrating a test function against a centered, tensor-valued function $u$
can be expressed as the integral of the gradient of the test function
against  $\nabla(-L)^{-1}u$.

\medskip
We will consider an eigenfunction of the Laplace
operator $L$ associated with $\mu$, and perform many iterations of
``apply $(-L)^{-1}$ to the centered gradient of the function''.
As in (\ref{eq_1513}) above, let $f\in\cA$ with $\|f\|_{L^2(\mu)}=1$ satisfy
$$
 Lf=-\lambda_1f.
$$
In particular, $\int_{\RR^n}f\,d\mu=0$. Set $h_1=f$.
For $i\geq1$, in the case where $h_i$ is a linear function
(including the zero function) we define $h_{i+1} = 0$ and $\alpha_{i+1} = 0$.
In the case where $h_i$ is {\it not} a linear function, we define
\begin{equation}
 h_{i+1}=\alpha_{i+1}(-L)^{-1}\nabla_0 h_i,
 \label{eq_1049}
\end{equation}
where $\nabla_0 u=\nabla u-\int_{\RR^n}\nabla u\,d\mu$ is the centered
gradient defined in \eqref{eq_1045}, and where
\begin{equation}
 \alpha_{i+1}
 =\frac{\| \nabla_0 h_i\|_2}
        {\|\nabla(-L)^{-1}\nabla_0 h_i\|_2}
\label{eq_alpha_normalization} \end{equation}
is a normalizing coefficient.
The fact that $h_{i+1}$ is well-defined
via \eqref{eq_1049}
follows from Lemma~\ref{lem_poisson}.
Note that $h_{i+1}(x)$ is a tensor with $i$ indices for any $x \in \RR^n$.
Unfortunately, it is not a symmetric tensor in general.

\medskip Note that when $\alpha_{i+1} > 0$, by (\ref{eq_poisson_energy}) and (\ref{eq_alpha_normalization}),
\begin{equation} \alpha_{i+1}^2\geq\lambda_1.
\label{eq_1529} \end{equation}
In fact, we remark that $\alpha_{i+1}$ equals
the ratio of the $L^2(\mu)$-norm and the $H^{-1}(\mu)$-norm of $\nabla_0 h_i$, where for $u \in \cA_0$,
$$ \| u \|_{H^{-1}(\mu)} = \sqrt{ \langle (-L)^{-1} u, u \rangle_{L^2(\mu)} } = \sup \left \{ \int_{\RR^n} \langle u, v \rangle d\mu \, ; \, \int_{\RR^n} |\nabla v|^2 d \mu \leq 1 \right\}. $$
In the last supreumum, the tensor-valued function $v$ on $\RR^n$ has the same number of indices as $u$.

Later on, we will mostly be concerned with the case where $\alpha_{i+1}^2$ differs from $\lambda_1$ by a factor
of at most $2$. The choice of the normalizing constant $\alpha_{i+1}$ ensures that
\begin{equation}
\| \nabla h_{i+1} \|_2 = \| \nabla_0 h_i \|_2 = \left \| \nabla h_i \, - \, \int_{\RR^n} (\nabla h_i) d \mu \right \|_2. \label{eq_1059}
\end{equation}
In order to study the decay of the Dirichlet energy of $h_i$, we denote
\begin{equation}
 v_i=\|\nabla h_i\|_2^2 \label{eq_1946} \end{equation}
 so that $v_1 = \|\nabla f \|_2^2 = \lambda_1$.
 Since centering can only decrease the $L^2(\mu)$-norm, we see from (\ref{eq_1059}) that the sequence $(v_i)$ is decreasing.
In fact,
\begin{equation}  p_i := v_{i} - v_{i+1} = \left|\int_{\RR^n}\nabla h_i\,d\mu\right|^2.
\label{eq_1449} \end{equation}
Thus, for any $N \geq 1$,
\begin{equation}
\lambda_1 = v_{N+1} + \sum_{i=1}^{N} p_i.
\label{eq_1124} \end{equation}
Later on we will show that the Dirichlet energy $v_{N+1}$
tends to zero as $N \to \infty$, thus it is the sum in (\ref{eq_1124}) that will be used to estimate $\lambda_1$.
 The non-symmetry of the tensor $\nabla_0 h_{i+1}$ will be dealt with
 in the following two lemmas:

\begin{lemma} For $i \geq 1$ set
\begin{equation}
 \chi_i
 =\|\nabla^2h_i-\alpha_{i+1}\nabla h_{i+1}\|_2^2.
\label{eq_1149} \end{equation}
Then for $i \geq 1$
there exists a tensor-valued function $\tilde{T}$ on $\RR^n$, symmetric in its first $2$ indices, such that
\begin{equation}
\left\|\nabla_0 h_{i+1}- \tilde{T} \right\|_2 \leq \sqrt{ \frac{\chi_i}{\lambda_1} }.
\label{eq_1644} \end{equation}
\label{lem_1651}
\end{lemma}

\begin{proof}
 Note that $\nabla^2 h_i = \nabla( \nabla_0 h_i)$
 is symmetric in its first $2$ indices, since it is a Hessian.
By  the integration by parts
(\ref{eq_ibp}) and by (\ref{eq_1049}) and (\ref{eq_1059}),
\begin{align*}
 \int_{\RR^n}
   \langle \nabla^2h_i,\nabla h_{i+1}\rangle\,d\mu
 =-\int_{\RR^n}\langle \nabla_0 h_i,Lh_{i+1}\rangle\,d\mu
 =\int_{\RR^n}\langle \nabla_0 h_i, \alpha_{i+1} \nabla_0 h_i \rangle\,d\mu = \alpha_{i+1}\|\nabla h_{i+1} \|_2^2.
 \end{align*}
Thus the orthogonal projection in $L^2(\mu)$ of $\nabla^2h_i$ on the subspace
spanned by $\nabla h_{i+1}$ equals $\alpha_{i+1} \nabla h_{i+1}$.
Recall from (\ref{eq_1946}) that $\| \nabla h_{i+1}\|_2^2 = v_{i+1}$.
The definition (\ref{eq_1149}) and the orthogonality thus yield
\begin{equation}
 \|\nabla^2h_i\|_2^2
 =\alpha_{i+1}^2v_{i+1}+\chi_i.
 \label{eq_hessian_decomposition}
\end{equation}
If $\alpha_{i+1} > 0$ then subtracting the mean in (\ref{eq_1149})  gives
\begin{equation}
\left\|\frac{\nabla_0(\nabla h_i)}{\alpha_{i+1}} - \nabla_0 h_{i+1}  \right\|_2^2
\leq \left\|\frac{\nabla^2 h_i}{\alpha_{i+1}} - \nabla h_{i+1} \right\|_2^2
 =\frac{\chi_i}{\alpha_{i+1}^2}
 \leq\frac{\chi_i}{\lambda_1},
 \label{eq_centered_poisson_error}
\end{equation}
where centering decreases the $L^2(\mu)$ norm, and
$\alpha_{i+1}^2\geq\lambda_1$ by (\ref{eq_1529}). The tensor $\nabla_0(\nabla h_i) = \nabla^2 h_i - \int \nabla^2 h_i d \mu$
is symmetric in its first two indices, and the lemma follows when $\alpha_{i+1} > 0$. If $\alpha_{i+1}=0$, then $\nabla_0 h_{i+1}=0$ and we may take $\tilde T=0$.
\end{proof}
\medskip
The preceding lemma controls the failure of symmetry in the
first two indices. Let us extend this control to the first
$q$ indices.

\begin{lemma}
\label{lem_block_symmetry}
Let $2\leq q\leq i$, and assume that
$0<\alpha_j^2\leq2\lambda_1$
for $i-q+2\leq j\leq i$.
Then
\begin{equation}
 \|\nabla_0 h_i-\cS_q\nabla_0 h_i\|_2^2
 \leq\frac{C^q}{\lambda_1}
       \sum_{j=i-q+1}^{i-1}\chi_j,
 \label{eq_block_symmetry}
\end{equation}
where $C\geq1$ is a universal constant.
Here $\cS_q$ symmetrizes the first $q$ indices of $\nabla_0 h_i$,
leaving the remaining indices fixed.
\end{lemma}

\begin{proof}
\medskip
Let $\tau_\ell$ exchange tensor positions $\ell$ and $\ell+1$.
By Lemma~\ref{lem_1651}, for each relevant $j$, the tensor
$\nabla_0 h_j$ is within distance $\sqrt{\chi_{j-1}/\lambda_1}$
of a tensor-valued function fixed by $\tau_1$. Hence
$$
 \|\nabla_0 h_j-\tau_1\nabla_0 h_j\|_2^2
 \leq\frac{4\chi_{j-1}}{\lambda_1}.
$$
For exchanging other indices, we will use the Poisson equation
(\ref{eq_1049}) which implies that for $2 \leq \ell < j$,
\begin{equation}
 \nabla_0 h_j-\tau_\ell \nabla_0 h_j
 =\alpha_j\nabla_0(-L)^{-1}
   \left(\nabla_0 h_{j-1}-\tau_{\ell-1}\nabla_0 h_{j-1}\right).
 \label{eq_1608}
\end{equation}
Indeed, $\nabla_0$ adds a new first index, whereas $(-L)^{-1}$
acts componentwise. By (\ref{eq_inverse_gradient}) and the
assumption that $\alpha_j^2 \leq 2 \lambda_1$, the operator $\alpha_j\nabla_0(-L)^{-1}$
has norm at most $\sqrt2$ in $L^2(\mu)$. Iterating $\ell-1$ times gives
\begin{equation}
 \|\nabla_0 h_i-\tau_\ell \nabla_0 h_i\|_2^2
 \leq  2^{\ell-1} \|\nabla_0 h_{i-\ell+1}  -  \tau_1\nabla_0 h_{i-\ell+1} \|_2^2
 \leq \frac{2^{\ell+1}}{\lambda_1}\chi_{i-\ell}
 \qquad(1\leq\ell<q).
 \label{eq_adjacent_indices}
\end{equation}
Every permutation $\pi$ of the first $q$ positions can be written as
$\pi=\tau_{\ell_1}\cdots\tau_{\ell_m}$ with $m\leq q^2$ and
$1\leq\ell_r<q$. Thus we have the  telescoping identity
$$
 \nabla_0 h_i-\pi \nabla_0 h_i
 =
 \sum_{r=1}^m
 \bigl(\tau_{\ell_1}\cdots\tau_{\ell_{r-1}}\bigr)
 \bigl(\nabla_0 h_i-\tau_{\ell_r}\nabla_0 h_i\bigr),
$$
where the empty product is the identity.
Since permutations of the indices preserve the norm, the Cauchy-Schwarz inequality
and $m\leq q^2$ give
\begin{equation}
 \|\nabla_0 h_i-\pi \nabla_0 h_i\|_2^2
 \leq m\sum_{r=1}^m
          \|\nabla_0 h_i-\tau_{\ell_r}\nabla_0 h_i\|_2^2
 \leq q^4\sum_{\ell=1}^{q-1}
          \|\nabla_0 h_i-\tau_\ell \nabla_0 h_i\|_2^2.
\label{eq_2007} \end{equation}
Since $\cS_q\nabla_0 h_i$ is the average over all $\pi$ of
the tensor $\pi \nabla_0 h_i$, by (\ref{eq_adjacent_indices}) and (\ref{eq_2007}),
$$
 \|\nabla_0 h_i-\cS_q\nabla_0 h_i\|_2^2
 \leq\frac{q^4\,2^q}{\lambda_1}
       \sum_{j=i-q+1}^{i-1}\chi_j.
$$
This proves (\ref{eq_block_symmetry}).
\end{proof}

The following lemma is the only place in the proof of Theorem \ref{thm_1032}
where log-concavity is substantially used.

\begin{lemma}
\label{lem_bochner_step}
For $i\geq1$,
\begin{equation}
 \|Lh_i\|_2^2 \geq \|Lh_{i+1}\|_2^2+\chi_i+a v_i,
 \label{eq_bochner_step}
\end{equation}
where $a>0$ is as in (\ref{eq_2256}).
\end{lemma}

\begin{proof}
The expression $\alpha_{i+1}^2 v_{i+1}$ appearing in (\ref{eq_hessian_decomposition})
has an additional meaning. By (\ref{eq_1059}) and (\ref{eq_1946}),
\begin{equation}
\alpha_{i+1}^2v_{i+1} =
\|\alpha_{i+1} \nabla_0 h_i \|_2^2 = \|Lh_{i+1}\|_2^2, \label{eq_1512} \end{equation}
and hence by (\ref{eq_hessian_decomposition}),
\begin{equation}   \|\nabla^2h_i\|_2^2
 = \|Lh_{i+1}\|_2^2 +\chi_i. \label{eq_1236} \end{equation}
 Applying the Bochner identity (\ref{eq_bochner}) to each
component of $h_i$ and using $\nabla^2V\geq a\,\id$, we obtain
$$
 \|Lh_i\|_2^2
 \geq\|\nabla^2h_i\|_2^2+a\|\nabla h_i\|_2^2
 =\|\nabla^2h_i\|_2^2+a v_i.
$$
Thus (\ref{eq_bochner_step}) follows from (\ref{eq_1236}).
\end{proof}

\begin{corollary} We have $\displaystyle \lim_{N \to \infty} v_N = 0$.
\label{cor_1454}
\end{corollary}

\begin{proof}
Summing (\ref{eq_bochner_step}) over $i=1,\ldots,N$ gives
$$ \|Lh_1\|_2^2 \geq
 \|Lh_{N+1}\|_2^2+\sum_{i=1}^N\chi_i
 +a\sum_{i=1}^N v_i.
$$
All terms on the right-hand side are nonnegative while $a > 0$. Therefore $\sum_{i=1}^{\infty} v_i$ is a convergent
series and  $v_N\to0$.
\end{proof}

Recall the tilted measures $\mu_z$ and the function $F_f$ defined in
(\ref{eq_1314}) and (\ref{eq_1315}) for $f\in L^2(\mu)$.
The Laplace operator $L_z$ associated with $\mu_z$ is related to the
Laplace operator $L$ associated with $\mu$ by
$$
 L_z u = Lu + z\cdot\nabla u.
$$
This relation leads to the following elegant
identity: for every $u\in\cA$ and $z \in \RR^n$,
\begin{equation}
F_{-Lu}(z) = \langle z,F_{\nabla u}(z)\rangle.
\label{eq_1335}
\end{equation}
Indeed, by integration by parts,
\begin{align*}
 F_{-Lu}(z)
 & = -\int_{\RR^n} Lu(x)\,d\mu_z(x)
 \\ & = - \int_{\RR^n} L_z u(x)\,d\mu_z(x)
    +\int_{\RR^n}z\cdot\nabla u(x)\,d\mu_z(x)
  = \langle z,F_{\nabla u}(z)\rangle,
\end{align*}
since $\int_{\RR^n}L_z u\,d\mu_z=0$. See also \cite[Section 4]{KlartagLog}.

\medskip Recall the Taylor tensor $\cT_d f$ from (\ref{eq_1337}).
For tensor-valued functions $u$ on $\RR^n$, we define~$F_u$ and~$\cT_d u$
componentwise, placing the $d$ Taylor indices before the original
tensor indices of $u$.
Recall from Section \ref{sec_background} that $\cS_qT$ denotes the average of a tensor $T$
over all permutations of its first $q$ indices.

\begin{lemma}
\label{lem_taylor_ibp}
Let $u \in \cA$ be a function, either scalar-valued or tensor-valued. Then,
\begin{equation}
 \cT_1(-Lu)=\int_{\RR^n}\nabla u\,d\mu,
 \label{eq_taylor_ibp_first}
\end{equation}
and for $d \geq 1$,
\begin{equation}
 \cT_{d+1}(-Lu) = \cS_{d+1}\cT_d(\nabla u).
 \label{eq_taylor_ibp}
\end{equation}
\end{lemma}

\begin{proof}
Apply (\ref{eq_1335}) to each component of $u$.
Differentiating $d+1$ times at $z=0$ and using the product rule gives
\begin{equation}
 \frac{1}{(d+1)!}\nabla_z^{d+1}F_{-Lu}(0)
 = \frac{1}{(d+1)!}\nabla_z^{d+1} \left \langle z,F_{\nabla u}(z) \right \rangle (0) =
 \frac{1}{d!}
 \cS_{d+1}\left(
 \nabla_z^d F_{\nabla u}(0)
 \right).
\label{eq_1356} \end{equation}
Indeed, exactly one derivative must fall on the factor $z$:
the term in which no derivative falls on it vanishes at $z=0$,
and terms with two or more derivatives on it vanish identically.
Thus, the entry corresponding to $i_1,\ldots,i_{d+1}$ in the middle tensor in (\ref{eq_1356}) equals
$$
 \frac{1}{(d+1)!}\sum_{k=1}^{d+1}
 \left(\prod_{\substack{1\leq j\leq d+1\\j\neq k}}
 \frac{\partial}{\partial z_{i_j}}\right)
 F_{\partial u/\partial x_{i_k}}(0),
$$
which is the corresponding entry of the tensor on the right-hand side of (\ref{eq_1356}).
Thus (\ref{eq_1356}) holds true.
Now (\ref{eq_taylor_ibp_first}) and (\ref{eq_taylor_ibp}) follow from (\ref{eq_1356}) and the definition
of the Taylor tensors.
\end{proof}

Lemma~\ref{lem_taylor_ibp} and (\ref{eq_1049}) imply the following:

\begin{corollary}
For $i \geq 1$ and $d \geq 1$,
\begin{equation}
 \cS_{d+1}\cT_d(\nabla_0 h_{i+1}) = \cT_{d+1}(-L h_{i+1})
 =\alpha_{i+1} \cT_{d+1}(\nabla_0 h_{i}).
 \label{eq_taylor_step}
\end{equation}
\label{cor_2042}
\end{corollary}

Let $N\geq1$ be the minimal index with
\begin{equation}
 v_{N+1}<\lambda_1/2.
 \label{eq_1459}
\end{equation}
Since $v_1=\lambda_1$, such an index exists by
Corollary~\ref{cor_1454}. For $2\leq i\leq N$, the minimality of $N$ gives
$v_i\geq\lambda_1/2$ and $\alpha_i > 0$, while (\ref{eq_bochner_step}) gives
$\|Lh_i\|_2^2\leq\|Lh_1\|_2^2=\lambda_1^2$.
Hence, by (\ref{eq_1512}), for $i=2,\ldots,N$,
\begin{equation}
\alpha_i^2=\frac{\|Lh_i\|_2^2}{v_i}\leq2\lambda_1.
 \label{eq_alpha_bounds}
\end{equation}
Corollary \ref{cor_2042}
connects Taylor coefficients
of degree $d$ of $F_{\nabla_0 h_i}$ with coefficients of degree $d+1$
of $F_{\nabla_0 h_{i-1}}$. Combining its iteration with
Lemma~\ref{lem_block_symmetry}, we obtain the following
estimate, connecting Taylor coefficients of degree $d$ of
$F_{\nabla_0h_i}$ with Taylor coefficients of degree $2d$ of
$F_{\nabla_0h_{i-d}}$.

\begin{lemma}
\label{lem_taylor_doubling}
If $d \geq 1$ and  $2d\leq i\leq N$ then
\begin{equation}
 |\cT_d(\nabla_0 h_i)|^2
 \leq(C\lambda_1)^d|\cT_{2d}(\nabla_0 h_{i-d})|^2
 +\frac{C^dR^{2d}}{\lambda_1}
   \sum_{j=i-2d+1}^{i-1}\chi_j,
 \label{eq_taylor_doubling}
\end{equation}
where $C \geq 1$ is a universal constant.
\end{lemma}
\begin{proof}
Set $T=\cT_d(\nabla_0 h_i)$. The first  $d$ indices of this tensor are Taylor indices,
and its remaining $i$ indices are ``derivative'' indices.
Let us first estimate the norm of the partially-symmetrized tensor $\cS_{2d}T$.
Since $\cS_{r+1}\cS_r=\cS_{r+1}$, repeated application
of (\ref{eq_taylor_step}) gives
\begin{equation}
 \cS_{d+k}T
 =\left(\prod_{j=i-k+1}^{i}\alpha_j\right)
   \cT_{d+k}(\nabla_0 h_{i-k})
 \qquad(1\leq k\leq d).
 \label{eq_1539}
\end{equation}
Taking $k=d$ and using (\ref{eq_alpha_bounds}), we obtain
\begin{equation}
 |\cS_{2d}T|^2
 \leq(2\lambda_1)^d|\cT_{2d}(\nabla_0 h_{i-d})|^2.
 \label{eq_iterated_taylor}
\end{equation}
In order to recover $T$ from its partial symmetrization, set
$$
 \widetilde T=\cT_d(\cS_{2d}\nabla_0 h_i).
$$
Here we symmetrize the first $2d$ indices of $\nabla_0 h_i$
before applying $\cT_d$. Thus $\widetilde T$ is symmetric
in its first $d$ indices and, separately, in its
next $2d$ indices.
By linearity of $\cT_d$,
$$
 T-\widetilde T=\cT_d(\nabla_0 h_i-\cS_{2d}\nabla_0 h_i).
$$
Set $u=\nabla_0 h_i-\cS_{2d}\nabla_0 h_i$, and write $u=(u_J)$, where
$J\in\{1,\ldots,n\}^i$ runs over the tensor's components.
Applying the Taylor bound (\ref{eq_1736}) to each scalar
function $u_J$ gives  $$ |\cT_d u_J|^2\leq R^{2d}\|u_J\|_2^2. $$
Summing over $J$, we obtain
$$
 |\cT_d u|^2
 =\sum_J|\cT_d u_J|^2
 \leq R^{2d}\sum_J\|u_J\|_2^2
 =R^{2d}\|u\|_2^2.
$$
Thus, by using
Lemma~\ref{lem_block_symmetry} with $q=2d$,
\begin{align}
 |T-\widetilde T|^2 = |\cT_d u|^2  \leq R^{2d}\|u\|_2^2 = R^{2d}\|\nabla_0 h_i-\cS_{2d}\nabla_0 h_i\|_2^2
 \leq\frac{C^{2d}R^{2d}}{\lambda_1}
       \sum_{j=i-2d+1}^{i-1}\chi_j.
 \label{eq_approximate_symmetry}
\end{align}
We may now apply Lemma~\ref{lem_tensor_recovery} for the tensor $\widetilde T$,
which satisfies its requirements. Together with the contraction property
of $\cS_{2d}$, this yields
\begin{align}
 |T| \leq |T-\widetilde T| + |\widetilde{T}|
 \leq|T-\widetilde T|+4^d|\cS_{2d}\widetilde T|. \label{eq_1711}
\end{align}
The conclusion of the lemma follows from (\ref{eq_iterated_taylor}), (\ref{eq_approximate_symmetry}) and (\ref{eq_1711}),
after increasing the value of the universal constant $C$.
\end{proof}

Recall that $p_i = v_i - v_{i+1}$ and that by our choice of $N$, the sum $\sum_{i=1}^N p_i$ has the order of magnitude
of $\lambda_1$. The following lemma will be used to compare this sum with $\lambda_1^2$.

\begin{lemma}
We have
\begin{equation}
 \sum_{i=1}^N p_i
 \leq\lambda_1^2+
 2\lambda_1\sum_{i=1}^{N-1}|\cT_1(\nabla_0 h_i)|^2.
 \label{eq_loss_first_taylor}
\end{equation}
\end{lemma}

\begin{proof}
Recall that $f=h_1$ satisfies $Lf=-\lambda_1f$ and $\|f\|_2=1$.
Integrating by parts using (\ref{eq_ibp}) against linear functions,
we obtain, as in \cite[Equation (2.10)]{KlartagLog},
$$
 p_1=\left|\int_{\RR^n}\nabla f(x)\,d\mu(x)\right|^2
 =\lambda_1^2\left|\int_{\RR^n}xf(x)\,d\mu(x)\right|^2
 \leq\lambda_1^2\|\cov(\mu)\|_{op}
 \leq\lambda_1^2.
$$
For $i\geq2$, Lemma~\ref{lem_taylor_ibp} and the definitions
(\ref{eq_1049}) and (\ref{eq_1449}) give
\begin{equation}
 p_i=\left|\int_{\RR^n}\nabla h_i\,d\mu\right|^2
 =|\cT_1(-Lh_i)|^2
 =\alpha_i^2|\cT_1(\nabla_0 h_{i-1})|^2.
 \label{eq_1450}
\end{equation}
Now (\ref{eq_loss_first_taylor})
follows by substituting the bound (\ref{eq_alpha_bounds}) into (\ref{eq_1450}), summing over
$i=2,\ldots,N$,  and adding the bound for $p_1$.
\end{proof}

\begin{proof}[Proof of Theorem~\ref{thm_1032}]
Recall the choice of $N$ in (\ref{eq_1459}). By
(\ref{eq_1124}) and by summing (\ref{eq_bochner_step}), we have
\begin{equation}
 \sum_{i=1}^N p_i
 =\lambda_1-v_{N+1}>\frac{\lambda_1}{2},
 \qquad
 \sum_{i=1}^N\chi_i\leq \| L h_1 \|_2^2 = \lambda_1^2.
 \label{eq_half_energy}
\end{equation}
Let $C\geq1$ be the constant from Lemma~\ref{lem_taylor_doubling}.
For each integer $d\geq1$, set
$$
 S_d=\sum_{i=1}^{N-d} |\cT_d(\nabla_0 h_i)|^2,
$$
where an empty sum is zero. In particular,
$S_d=0$ for $d\geq N$.
We apply the Taylor bound (\ref{eq_1736}) component by component, and obtain
by (\ref{eq_1059}) and (\ref{eq_1946}),
$$
 |\cT_d(\nabla_0 h_i)|^2
 \leq R^{2d}\|\nabla_0 h_i\|_2^2
 =R^{2d}v_{i+1}\leq\lambda_1R^{2d}.
$$
Thus the indices $i<2d$ contribute at most
$(2d-1)\lambda_1R^{2d}$ to $S_d$.
For the remaining indices, sum (\ref{eq_taylor_doubling}) over
$2d\leq i\leq N-d$. Thus,
\begin{align}
  \label{eq_summed_taylor_doubling}
   S_d
 &\leq (2d-1)\lambda_1 R^{2d} +
 \sum_{i=2d}^{N-d} \left( (C\lambda_1)^d |\cT_{2d}(\nabla_0 h_{i-d})|^2 + \frac{C^dR^{2d}}{\lambda_1}
   \sum_{j=i-2d+1}^{i-1}\chi_j \right)
   \\ & \leq (2d-1)\lambda_1R^{2d} + (C\lambda_1)^dS_{2d}
       +\frac{2d \cdot C^dR^{2d}}{\lambda_1}
          \sum_{j=1}^N\chi_j \leq(C\lambda_1)^dS_{2d}
       +4dC^d\lambda_1R^{2d}. \notag
\end{align}
Here, we used the fact that  each
$\chi_j$ occurs at most $2d$ times, as well as the bound (\ref{eq_half_energy}).

\medskip
Assume by contradiction that
\begin{equation}
 \lambda_1R^2\leq\frac{1}{100C^2}.
\label{eq_2145} \end{equation}
Thus $C \lambda_1 \leq 1/100$ as $R,C \geq 1$.
Multiplying (\ref{eq_summed_taylor_doubling}) by
$(C\lambda_1)^{d-1}$ yields
\begin{align*}
 (C\lambda_1)^{d-1}S_d
 \leq(C\lambda_1)^{2d-1}S_{2d}
       +4dC^{2d-1}(\lambda_1R^2)^d
 \leq(C\lambda_1)^{2d-1}S_{2d}
       +\frac{4d}{100^d}.
\end{align*}
Choose an integer $m\geq1$ with $2^m\geq N$, and sum these
inequalities over $d=1,2,4,\ldots,2^{m-1}$.
Since $S_{2^m} = 0$, we have
$$
 S_1
 \leq4\sum_{k=0}^{m-1}\frac{2^k}{100^{2^k}}
 \leq4\sum_{d=1}^{\infty}\frac{d}{100^d}
 <\frac{1}{20}.
$$
Since $R,C\geq1$, our assumption (\ref{eq_2145}) also gives $\lambda_1\leq1/100$.
It now follows from (\ref{eq_loss_first_taylor}) that
$$
 \sum_{i=1}^N p_i
 \leq\lambda_1^2+2\lambda_1S_1
 <\lambda_1\left(\frac{1}{100}+\frac{1}{10}\right)
 <\frac{\lambda_1}{2},
$$
contradicting (\ref{eq_half_energy}).
Therefore $\lambda_1R^2>(100C^2)^{-1}$, and
$$
 C_P(\mu)=\lambda_1^{-1}\leq100C^2R^2.
$$
\end{proof}

\begin{remark} The bound (\ref{eq_1736}) on the Taylor coefficients
implies that $F_f(z)$ admits an analytic continuation to the complex ball of radius $1/R$ centered at the origin in $\CC^n$.
One might wonder if there is a complex-analytic proof of Theorem \ref{thm_1032}. We do not know if such a proof exists.
\end{remark}

\section{Cumulants and stochastic localization}
\label{sec_cumulants}

Let $\mu$ be a log-concave probability measure on $\RR^n$. Its
Laplace transform is finite in a neighborhood of the origin, and we
write, as in Section~\ref{sec_taylor},
$$
 \Lambda_\mu(z)=\log\int_{\RR^n}e^{\langle z,x\rangle}\,d\mu(x).
$$
The {\it cumulant tensor of order $m$} is
$$
 \kappa_m^\mu=\nabla^m\Lambda_\mu(0) \in (\RR^n)^{\otimes m}.
$$
We identify this symmetric tensor with the $m$-linear form $\kappa_m(\cdot,\ldots,\cdot)$. Thus
$\kappa_1^\mu \in \RR^n$ is the barycenter, $\kappa_2^\mu \in \RR^{n \times n}$ is the covariance
matrix, and $\kappa_3^\mu$ is the third centered moment tensor.
For $u\in\RR^n$, abbreviate
$\kappa_m^\mu(u)=\kappa_m^\mu(u,\cdot,\ldots,\cdot)$, and define
the positive semidefinite symmetric matrix $R_m^\mu \in \RR^{n \times n}$ by
\begin{equation}
 \langle R_m^\mu u,u\rangle
 =|\kappa_m^\mu(u)|^2
 =\sum_{i_2,\ldots,i_m=1}^n
   \kappa_m^\mu(u,e_{i_2},\ldots,e_{i_m})^2.
 \label{eq_cumulant_matrix}
\end{equation}
Here $e_1,\ldots,e_n$ are the standard basis vectors. Our goal in this section and the next two sections is to prove the following:

\begin{theorem}
\label{thm_cumulants}
There exists a universal constant $K\geq1$ such that, for every
isotropic log-concave probability measure $\mu$ on $\RR^n$ and every
integer $m\geq2$,
\begin{equation}
 R_m^\mu\leq K^{m-1}((m-1)!)^2\,\id.
 \label{eq_cumulant_bound}
\end{equation}
\end{theorem}

The case $m=2$ is immediate, since $R_2^\mu=\id$. The case $m=3$
is a difficult theorem, which follows from the recent solution of
the generalized thin-shell conjecture. The latter is obtained by
an anisotropic extension of the moment-measure argument of Chen
and Klartag \cite{ChenKlartag}; see, for example,
Bizeul \cite[Section~2]{BizeulMM} or Letwin \cite{letwin}. We record the third-cumulant
bound in the following form.

\begin{lemma}
\label{lem_third_cumulant}
Every isotropic log-concave probability measure $\mu$ satisfies
$R_3^\mu\leq8\,\id$.
\end{lemma}

\begin{proof}
Let $X$ have law $\mu$. The quadratic-form estimate
\cite[Theorem~1.2]{letwin} gives, for every symmetric matrix $M$,
$$
 \var(\langle MX,X\rangle)
 \leq 2\EE|\nabla\langle MX,X\rangle|^2
 =8|M|^2.
$$
Since $X$ is centered and isotropic, the Cauchy--Schwarz inequality gives
$$
 |\langle\kappa_3^\mu(u),M\rangle|^2
 =|\cov(\langle X,u\rangle,\langle MX,X\rangle)|^2
 \leq8|u|^2|M|^2.
$$
Taking the supremum over all symmetric matrices $M$ with $|M|=1$ concludes
the proof.
\end{proof}

\medskip
The study of $R_3^\mu$ was motivated by its role in stochastic
localization, which we now introduce. This process was introduced
by Eldan \cite{eldan}. A slightly different version was later used
and popularized by Lee and Vempala \cite{LV}. We use Eldan's original process. One
advantage of this choice is the simple evolution of the mean
covariance, given in (\ref{eq_mean_covariance}) below.

\medskip
Fix an isotropic, log-concave probability measure $\mu$ in $\RR^n$. 
By a standard approximation argument, in proving 
Theorem \ref{thm_cumulants} we may assume that $\mu$ is compactly-supported. 
This ensures that the process
below is well-defined, that a strong solution to the stochastic differential equation exists and is unique, and that the moments of $\mu_t$ of all orders
remain bounded; see \cite{eldan} and the discussion in
\cite[Section~2.1 and Appendix~A]{Chen}.
Let $(B_t)_{t\geq0}$ be a standard Brownian motion in $\RR^n$ with $B_0 = 0$.
Following \cite{eldan}, set
\begin{equation}
 \frac{d\mu_t}{d\mu}(x)
 =\frac1{Z_t}\exp\left(\langle c_t,x\rangle
                      -\frac12\langle Q_tx,x\rangle\right),
 \qquad
 \begin{cases}
 dc_t=A_t^{-1}a_t\,dt+A_t^{-1/2}\,dB_t,\\
 dQ_t=A_t^{-1}\,dt,
 \end{cases}
 \label{eq_localization}
\end{equation}
where $c_0=Q_0=0$, where $Z_t$ is a normalizing constant that makes $\mu_t$ a probability measure, and
$$
 a_t=\int_{\RR^n}x\,d\mu_t(x),\qquad A_t=\cov(\mu_t).
$$
Each probability measure $\mu_t$ is log-concave, since
$Q_t\geq0$.
The expectation sign $\EE$ below stands for expectation with respect to the Brownian motion; integrals
against the random measure $\mu_t$ are written explicitly.

\medskip
The density $p_t=d\mu_t/d\mu$ satisfies the stochastic differential equation
$$
 dp_t(x)=p_t(x)\langle A_t^{-1/2}(x-a_t),dB_t\rangle.
$$
Consequently, for every measurable function $f$ bounded on the support of $\mu$,
\begin{equation}
 d\left(\int_{\RR^n}f\,d\mu_t\right)
 =\left\langle A_t^{-1/2}
     \int_{\RR^n}f(x)(x-a_t)\,d\mu_t(x),dB_t\right\rangle.
 \label{eq_localization_test}
\end{equation}
Applying this identity to the first and second moments, and using
It\^o's formula for $a_t \otimes a_t$, gives the well-known equations
\begin{equation}
 \begin{gathered}
 da_t=A_t^{1/2}\,dB_t,\qquad
 dA_t=\sum_{i=1}^nH_{i,t}\,dB_{t,i}-A_t\,dt,\\
 H_{i,t}=\kappa_3^{\mu_t}(\cdot,\cdot,A_t^{-1/2}e_i)
 =\int_{\RR^n}(x-a_t)(x-a_t)^\top
       \langle A_t^{-1/2}e_i,x-a_t\rangle\,d\mu_t(x).
 \end{gathered}
 \label{eq_covariance_dynamics}
\end{equation}
It follows that $(e^tA_t)_{t \geq 0}$ is a martingale, and hence
\begin{equation}
 \EE A_t=e^{-t}A_0=e^{-t}\,\id.
 \label{eq_mean_covariance}
\end{equation}
The Eldan process has a moment-generating property: the differential
of the cumulant of order $m$ involves the cumulant of order $m+1$.
We will use this property to propagate cumulant bounds inductively,
starting at order three.

\medskip
Set $\Lambda_t=\Lambda_{\mu_t}$ and $\kappa_m(t)=\kappa_m^{\mu_t}$,
and write
$$
 M_t(z)=\int_{\RR^n}e^{\langle z,x\rangle}\,d\mu_t(x).
$$
Applying (\ref{eq_localization_test}) gives
$$
 \begin{aligned}
 dM_t(z)
 &=\left\langle A_t^{-1/2}
       \int_{\RR^n}e^{\langle z,x\rangle}(x-a_t)\,d\mu_t(x),dB_t\right\rangle\\
 &=\left\langle A_t^{-1/2}(\nabla M_t(z)-M_t(z)a_t),dB_t\right\rangle.
 \end{aligned}
$$
Since $\Lambda_t=\log M_t$, It\^o's formula yields
\begin{equation}
 \begin{split}
 d\Lambda_t(z)
 ={}&\langle A_t^{-1/2}(\nabla\Lambda_t(z)-a_t),dB_t\rangle\\
 &-\frac12\left\langle\nabla\Lambda_t(z)-a_t,
             A_t^{-1}(\nabla\Lambda_t(z)-a_t)\right\rangle\,dt\\
 ={}&\langle A_t^{-1/2}v_t(z),dB_t\rangle
       -\frac12|A_t^{-1/2}v_t(z)|^2\,dt,
 \end{split}
 \label{eq_loglaplace_dynamics}
\end{equation}
where $$ v_t(z)=\nabla\Lambda_t(z)-a_t. $$
We differentiate this identity in $z$ to obtain the cumulant equations.
Write $[m]=\{1,\ldots,m\}$. Define an $m$-linear form $L_m(t)$ by
\begin{equation}
 \begin{split}
 L_m(t)(h_1,\ldots,h_m)
 =\sum_{\substack{J\subset[m],\ 1\in J\\2\leq|J|\leq m-2}}
 \Big\langle\kappa_{|J|+1}(t)(h_J,\cdot),
 A_t^{-1}\kappa_{m-|J|+1}(t)(h_{J^c},\cdot)\Big\rangle,
 \end{split}
 \label{eq_lower_cumulants}
\end{equation}
where for $J\subset[m]$ we write $h_J = (h_j)_{j \in J}$. 
Note that both $\kappa_{|J|+1}(t)(h_J,\cdot)$ and $\kappa_{m-|J|+1}(t)(h_{J^c},\cdot)$
are tensors of order $1$, i.e., they are vectors in $\RR^n$. 
Note also that (\ref{eq_lower_cumulants}) involves cumulants of order between $3$ and $m-1$.

\begin{lemma}
\label{lem_cumulant_dynamics}
For $m\geq3$,
\begin{equation}
 d\kappa_m(t)
 =\kappa_{m+1}(t)(\cdot,\ldots,\cdot,A_t^{-1/2}dB_t)
  -\bigl(m\kappa_m(t)+L_m(t)\bigr)\,dt.
 \label{eq_cumulant_dynamics}
\end{equation}
\end{lemma}

\begin{proof}
We need to compute the $m^{th}$-spatial derivatives at $z= 0$ of the expression on the third 
line of (\ref{eq_loglaplace_dynamics}).
Fix a time $t \geq 0$. 
Then $v_t(0)=0$, $\nabla v_t(0)=A_t$,
and, for $r\geq1$,
$$
 \langle\nabla^rv_t(0)[h_1,\ldots,h_r],w\rangle
 =\kappa_{r+1}(t)(h_1,\ldots,h_r,w).
$$
Thus, 
$$
 \left\langle A_t^{-1/2}\nabla^mv_t(0)[h_1,\ldots,h_m],dB_t\right\rangle
 =\kappa_{m+1}(t)(h_1,\ldots,h_m,A_t^{-1/2}dB_t).
$$
This gives the martingale coefficient in (\ref{eq_cumulant_dynamics}).
 For the drift, write
$$
 |A_t^{-1/2}v_t(z)|^2=\langle v_t(z),A_t^{-1}v_t(z)\rangle.
$$
Recall that $a_t$ and $A_t$ do not depend on $z$.
The Leibniz rule gives
\begin{equation}
 \begin{split}
 \nabla^m\langle v_t,A_t^{-1}v_t\rangle(0)[h_1,\ldots,h_m]
 =\sum_{\varnothing\neq J\subsetneq[m]}
 \Big\langle\kappa_{|J|+1}(t)(h_J,\cdot),
 A_t^{-1}\kappa_{m-|J|+1}(t)(h_{J^c},\cdot)\Big\rangle.
 \end{split}
 \label{eq_cumulant_product_rule}
\end{equation}
Indeed, each subset $J$ specifies which derivatives fall on the first factor;
the remaining derivatives fall on the second. The empty and full
subsets give zero because $v_t(0)=0$.

\medskip Let us analyze the summands in (\ref{eq_cumulant_product_rule}) in which $|J| \in \{1, m-1\}$.
When $m\geq3$, there are $2m$ such summands. 
Observe that each such summand equals $\kappa_m(t)(h_1,\ldots,h_m)$, since
$$
 \left\langle A_th_1,
       A_t^{-1}\kappa_m(t)(h_2,\ldots,h_m,\cdot)\right\rangle
 =\kappa_m(t)(h_1,\ldots,h_m).
$$
We group the other summands, with $2 \leq |J| \leq m-2$ in $L_m(t)$. Indeed, observe that they occur in equal pairs indexed by $J$ and $J^c$.
Keeping the representative containing $1$ gives
$L_m(t)(h_1,\ldots,h_m)$. This  proves (\ref{eq_cumulant_dynamics}).
\end{proof}

\medskip
For a tensor $T$ of order $r$, define a time-dependent norm by
\begin{equation}
 |T|_t^2
 =|T(A_t^{-1/2}\cdot,\ldots,A_t^{-1/2}\cdot)|^2
 =\langle T,(A_t^{-1})^{\otimes r}T\rangle.
 \label{eq_covariance_norm}
\end{equation}
Here, for $M \in \RR^{n \times n}$ and a tensor $T$ with $k$ indices,
we write $$ M^{\otimes k} T (h_1, \ldots, h_k) = T (M h_1, \ldots, M h_k). $$
Write $\langle\cdot,\cdot\rangle_t$ for the
corresponding inner product. Fix a deterministic unit vector
$u\in\RR^n$ and set
\begin{equation}
 \cE_m(t)=|\kappa_m(t)(u)|_t^2\qquad(m\geq2).
 \label{eq_cumulant_energy}
\end{equation}
That is, we 
compute the $| \cdot |_t$-norm
of the $(m-1)$-tensor $\kappa_m(t)(u)$.

\medskip
Let $\nu_t$ be the push-forward of $\mu_t$ under the affine map
$x\mapsto A_t^{-1/2}(x-a_t)$. Then $\nu_t$ is an isotropic, compactly
supported log-concave probability measure. The transformation rule
for cumulants is
$$
 \kappa_m^{\nu_t}(h_1,\ldots,h_m)
 =\kappa_m(t)(A_t^{-1/2}h_1,\ldots,A_t^{-1/2}h_m),
 \qquad m\geq2.
$$
In particular, for every $v\in\RR^n$,
\begin{equation}
 |\kappa_m(t)(v)|_t^2
 =|\kappa_m^{\nu_t}(A_t^{1/2}v)|^2
 \leq b\langle A_tv,v\rangle
 \label{eq_affine_cumulant}
\end{equation}
in the case where $b > 0$ is such that $R_m^\nu\leq b\,\id$ for all compactly supported isotropic
log-concave measures $\nu$. Using (\ref{eq_mean_covariance}) and
Lemma~\ref{lem_third_cumulant}, we have
\begin{equation}
 \begin{gathered}
 \cE_2(t)=\langle A_tu,u\rangle,\qquad \EE\cE_2(t)=e^{-t},\\
 \cE_3(t)\leq8\langle A_tu,u\rangle,\qquad
 \EE\cE_3(t)\leq8e^{-t}.
 \end{gathered}
 \label{eq_initial_energies}
\end{equation}

\medskip
We will also use the following matrix form of the same third-order
bound:
\begin{equation}
 \sum_{i=1}^n H_{i,t}A_t^{-1}H_{i,t}\leq8A_t.
 \label{eq_covariance_noise_bound}
\end{equation}
Indeed, for every $v\in\RR^n$, the quadratic form on the left is
$$
 \sum_{i,j=1}^n
 \kappa_3(t)(v,A_t^{-1/2}e_j,A_t^{-1/2}e_i)^2
 =|\kappa_3(t)(v)|_t^2\leq8\langle A_tv,v\rangle,
$$
by (\ref{eq_affine_cumulant}) and Lemma~\ref{lem_third_cumulant}.

\section{An inequality for cumulant energies}
\label{sec_energy}

We continue with the construction and notation of 
Section~\ref{sec_cumulants}. In particular, $u \in \RR^n$ is a fixed unit vector.
 The main estimate is the following
lower bound for the drift of $\cE_m$. Here and below, the drift of
an It\^o process denotes the coefficient of $dt$ in its It\^o differential.

\begin{lemma}
\label{lem_energy_drift}
For every $m\geq3$,
\begin{equation}
 \drift\cE_m(t)
 \geq\frac12\cE_{m+1}(t)-Cm^2\cE_m(t)
       -2\langle\kappa_m(t)(u),L_m(t)(u)\rangle_t,
 \label{eq_energy_drift}
\end{equation}
where $C\geq1$ is a universal constant.
\end{lemma}

\begin{proof}
Write
$$
 T_t=\kappa_m(t)(u),\qquad
 W_{i,t}=\kappa_{m+1}(t)(u,\cdot,\ldots,\cdot,A_t^{-1/2}e_i).
$$
These tensors have order $m-1$. Let us recall the dynamics of $A_t$ and $T_t = \kappa_m(t)$. That is, by
(\ref{eq_covariance_dynamics}) and Lemma~\ref{lem_cumulant_dynamics},
\begin{equation}
 \begin{cases}
 dA_t=\displaystyle\sum_{i=1}^nH_{i,t}\,dB_{t,i}-A_t\,dt,\\[3pt]
 dT_t=\displaystyle\sum_{i=1}^nW_{i,t}\,dB_{t,i}
             -(mT_t+L_m(t)(u))\,dt.
 \end{cases}
 \label{eq_energy_pair}
\end{equation}
For a positive definite symmetric matrix $A \in \RR^{n \times n}$ and a tensor $T$ of
order $m-1$, set
$$
 F(A,T)=\langle T,(A^{-1})^{\otimes(m-1)}T\rangle.
$$
Thus $\cE_m(t)=F(A_t,T_t)$, and It\^o's formula gives
\begin{equation}
 \begin{split}
 \drift\cE_m(t)
 ={}&DF(A_t,T_t)[-A_t,-mT_t-L_m(t)(u)]\\
 &+\frac12\sum_{i=1}^nD^2F(A_t,T_t)
       [(H_{i,t},W_{i,t}),(H_{i,t},W_{i,t})].
 \end{split}
 \label{eq_energy_ito}
\end{equation}
Here $DF$ and $D^2F$ denote the first and second differentials
with respect to both arguments of $F$.

\medskip
We will compute the derivatives of $F(A, T)$ at $A=\id$.
Let us first show how to reduce to this case by a linear change of variables.
For a fixed symmetric positive definite matrix $P$,
\begin{equation}
 F(A,T)=F(PAP,P^{\otimes(m-1)}T).
 \label{eq_energy_invariance}
\end{equation}
The identity follows from $P(PAP)^{-1}P=A^{-1}$.
For a symmetric matrix $H$ and a tensor $W$, the identity gives
\begin{equation}|
 F(A+\eps H,T+\eps W)
 =F(PAP+\eps PHP,P^{\otimes(m-1)}T+\eps P^{\otimes(m-1)}W).
\label{eq_1932} \end{equation}
Differentiating (\ref{eq_1932}) with respect to $\eps$ allows us to relate the derivative of $F$
at $(A,T)$ with the derivative at $(PAP,P^{\otimes(m-1)}T)$.

\medskip
Fix a time $t$, suppress time subscripts, and choose $P=A^{-1/2}$.
Set
\begin{equation}
 \begin{aligned}
 \widetilde T&=P^{\otimes(m-1)}T,&
 \widetilde H_i&=PH_iP,\\
 \widetilde W_i&=P^{\otimes(m-1)}W_i,&
 \widetilde L&=P^{\otimes(m-1)}L_m(t)(u).
 \end{aligned}
 \label{eq_normalized_coefficients}
\end{equation}
Since $PAP=\id$, equation (\ref{eq_energy_ito}) becomes
\begin{equation}
 \begin{split}
 \drift\cE_m(t)
 ={}&DF(\id,\widetilde T)[-\id,-m\widetilde T-\widetilde L]\\
 &+\frac12\sum_{i=1}^nD^2F(\id,\widetilde T)
       [(\widetilde H_i,\widetilde W_i),
        (\widetilde H_i,\widetilde W_i)].
 \end{split}
 \label{eq_normalized_ito}
\end{equation}

\medskip
For a symmetric matrix $H$, put
$$
 H^{(s)}=\id^{\otimes(s-1)}\otimes H\otimes\id^{\otimes(m-1-s)},
 \qquad
 \Sigma_H=\sum_{s=1}^{m-1}H^{(s)}.
$$
Thus $H^{(s)}$ acts on the $s$th input of a tensor of order $m-1$.
These operators are self-adjoint and they commute with one another. Expanding $(\id+\eps H)^{-1}$ in each tensor factor,
we obtain
\begin{equation}
 \begin{split}
 ((\id+\eps H)^{-1})^{\otimes(m-1)}
 ={}&\id-\eps\Sigma_H\\
 &+\eps^2\left(\sum_{s=1}^{m-1}(H^{(s)})^2
       +\sum_{1\leq r<s\leq m-1}H^{(r)}H^{(s)}\right)+O(\eps^3)\\
 ={}&\id-\eps\Sigma_H
   +\frac{\eps^2}{2}\left(\Sigma_H^2
                    +\sum_{s=1}^{m-1}(H^{(s)})^2\right)+O(\eps^3).
 \end{split}
 \label{eq_inverse_tensor_expansion}
\end{equation}
Consequently,
\begin{align} 
 \label{eq_joint_energy_expansion}
  F(\id+\eps H,\widetilde T+\eps W)
 & =|\widetilde T|^2
 +\eps\left(2\langle\widetilde T,W\rangle
            -\langle\widetilde T,\Sigma_H\widetilde T\rangle\right)\\
 &+\eps^2\left(|W|^2-2\langle W,\Sigma_H\widetilde T\rangle
          +\frac12|\Sigma_H\widetilde T|^2
          +\frac12\sum_{s=1}^{m-1}|H^{(s)}\widetilde T|^2\right)
 +O(\eps^3). \nonumber
 \end{align}
The coefficients of $\eps$ and $\eps^2$ give the first derivative
and half the second derivative, respectively. Plugging
(\ref{eq_joint_energy_expansion}) into (\ref{eq_normalized_ito}) and
using $\Sigma_{-\id}=-(m-1)\id$ yields
\begin{align}
 \drift\cE_m(t)
 ={}&-(m+1)|\widetilde T|^2
       -2\langle\widetilde T,\widetilde L\rangle
       +\sum_{i=1}^n|\widetilde W_i|^2
       -2\sum_{i=1}^n
          \langle\widetilde W_i,\Sigma_{\widetilde H_i}\widetilde T\rangle
       \notag\\
 &+\frac12\sum_{i=1}^n|\Sigma_{\widetilde H_i}\widetilde T|^2
       +\frac12\sum_{i=1}^n\sum_{s=1}^{m-1}
                      |\widetilde H_i^{(s)}\widetilde T|^2
       \label{eq_exact_energy_drift}\\
 \geq{}&\frac12\sum_{i=1}^n|\widetilde W_i|^2
       -(m+1)|\widetilde T|^2
       -2\langle\widetilde T,\widetilde L\rangle
       -2\sum_{i=1}^n|\Sigma_{\widetilde H_i}\widetilde T|^2.
       \notag
\end{align}
Here, the inequality was obtained  by 
discarding the last two nonnegative terms in the
equality and using the inequality
$2|\langle U,V\rangle|\leq|U|^2/2+2|V|^2$.

\medskip
Recall from (\ref{eq_normalized_coefficients}) 
that 
\begin{equation}
 |\widetilde T|^2=\cE_m(t),\qquad
 \langle\widetilde T,\widetilde L\rangle
     =\langle\kappa_m(t)(u),L_m(t)(u)\rangle_t,
 \qquad
 \sum_{i=1}^n|\widetilde W_i|^2=\cE_{m+1}(t).
 \label{eq_energy_identifications}
\end{equation}
Inequality (\ref{eq_covariance_noise_bound}) gives
$$
 \sum_{i=1}^n\widetilde H_i^2
 =P\left(\sum_{i=1}^nH_iA^{-1}H_i\right)P\leq8\,\id.
$$
Thus, for each $s$,
$$
 \sum_{i=1}^n|\widetilde H_i^{(s)}\widetilde T|^2
 =\left\langle\widetilde T,
       \left(\sum_{i=1}^n\widetilde H_i^2\right)^{(s)}
       \widetilde T\right\rangle
 \leq8|\widetilde T|^2.
$$
By using the Cauchy--Schwarz inequality for  $s=1,\ldots,m-1$ we obtain 
\begin{equation}
 \sum_{i=1}^n|\Sigma_{\widetilde H_i}\widetilde T|^2
 \leq(m-1)\sum_{s=1}^{m-1}\sum_{i=1}^n
                  |\widetilde H_i^{(s)}\widetilde T|^2
 \leq8(m-1)^2\cE_m(t).
 \label{eq_metric_error_bound}
\end{equation}
Substitution into (\ref{eq_exact_energy_drift}) proves
(\ref{eq_energy_drift}), since $m+1+16(m-1)^2\leq Cm^2$.
\end{proof}

Equation (\ref{eq_energy_drift}) allows for a certain inductive scheme on the paramater $m$,
which will be our next step.

\begin{lemma}
\label{lem_integrated_energy}
Let $m\geq3$, and assume that
$$
 \int_0^\infty\EE\cE_m(t)\,dt<\infty,
 \qquad
 \int_0^\infty\EE|L_m(t)(u)|_t^2\,dt<\infty.
$$
Then,
\begin{align} \nonumber
 \cE_m(0) & +\frac12\int_0^\infty\EE\cE_{m+1}(t)\,dt
 \\ & \leq Cm^2\int_0^\infty\EE\cE_m(t)\,dt +2\left(\int_0^\infty\EE\cE_m(t)\,dt\right)^{1/2}
       \left(\int_0^\infty\EE|L_m(t)(u)|_t^2\,dt\right)^{1/2}.
        \label{eq_integrated_energy}
 \end{align}
\end{lemma}

\begin{proof}
Taking expectations in (\ref{eq_energy_drift}) and integrating gives
\begin{align}
 \cE_m(0)+\frac12\int_0^T\EE\cE_{m+1}(t)\,dt
 \leq{}&\EE\cE_m(T)+Cm^2\int_0^T\EE\cE_m(t)\,dt\notag\\
 &+2\int_0^T\EE\langle\kappa_m(t)(u),L_m(t)(u)\rangle_t\,dt.
 \label{eq_finite_energy_integral}
\end{align}
By the Cauchy--Schwarz inequality, the last integral
is bounded in absolute value by
$$
 \left(\int_0^T\EE\cE_m(t)\,dt\right)^{1/2}
 \left(\int_0^T\EE|L_m(t)(u)|_t^2\,dt\right)^{1/2}.
$$
The integrability of $\EE\cE_m$ and the monotone convergence theorem 
allow us to let $T$ tend to infinity and conclude (\ref{eq_integrated_energy}).
\end{proof}

\section{The cumulant bound}
\label{sec_induction}

\begin{proof}[Proof of Theorem~\ref{thm_cumulants}]
Fix a dimension $n \geq 1$.
Choose a universal constant $K\geq1$, to be specified below, and set
\begin{equation}
 b_m=K^{m-1}((m-1)!)^2\qquad(m\geq2).
 \label{eq_induction_constants}
\end{equation}
We will prove that for every compactly-supported, isotropic,
log-concave measure $\mu$ in $\RR^n$, and for every unit vector $u \in \RR^n$,
\begin{equation}
 \cE_m(0)+\frac12\int_0^\infty\EE\cE_{m+1}(t)\,dt\leq b_m,
 \qquad \text{for} \ m\geq3.
 \label{eq_cumulant_induction}
\end{equation}

\medskip
Fix $m\geq 3$ and assume that (\ref{eq_cumulant_induction}) has been proven for all smaller $m$; the case $m =2$ follows immediately from 
(\ref{eq_initial_energies}).
Let $\mu$ be a compactly-supported, isotropic
log-concave measure $\mu$ in $\RR^n$. 
By the induction hypothesis, for any $2\leq r<m$,
\begin{equation}
\int_0^\infty\EE\cE_{r+1}(t)\,dt\leq2b_r
 \qquad \text{and} \qquad
  R_r^\nu\leq b_r\,\id,
 \label{eq_induction_inputs}
\end{equation}
where $\cE_{r+1}(t)$ is the cumulant energy associated with $\mu$,
and where $\nu$ is any compactly-supported, isotropic
log-concave probability measure in $\RR^n$. In particular,
\begin{equation}
 \int_0^\infty\EE\cE_m(t)\,dt\leq2b_{m-1}.
 \label{eq_current_energy_integral}
\end{equation}
In order to apply Lemma~\ref{lem_integrated_energy}, we need to estimate the time
integral of $\EE |L_m(t)(u)|_t^2$.

\medskip
Recall from (\ref{eq_lower_cumulants}) that, setting $h_1=u$,
\begin{equation}
 L_m(t)(u,h_2,\ldots,h_m)
 =\sum_{\substack{J\subset[m],\ 1\in J\\2\leq|J|\leq m-2}}
 \Big\langle\kappa_{|J|+1}(t)(h_J,\cdot),
 A_t^{-1}\kappa_{m-|J|+1}(t)(h_{J^c},\cdot)\Big\rangle.
\label{eq_2003} \end{equation}
We first bound each summand individually, using the induction bounds of \eqref{eq_induction_inputs}. 
The integral estimate from (\ref{eq_induction_inputs}) will be used for bounding the factor containing $u$, while the second factor will be bounded by the pointwise estimate
from (\ref{eq_induction_inputs}).

\medskip
For a given size $a=|J|$, the various summands in (\ref{eq_2003}) differ only by a permutation
of $h_2,\ldots,h_m$. We may therefore take $J=\{1,\ldots,a\}$,
where $2\leq a\leq m-2$, and write
\begin{equation}
 \begin{split}
 P_a(t)(h_2,\ldots,h_m)
 =\Big\langle\kappa_{a+1}(t)(u,h_2,\ldots,h_a,\cdot),
 A_t^{-1}\kappa_{m-a+1}(t)(h_{a+1},\ldots,h_m,\cdot)\Big\rangle.
 \end{split}
 \label{eq_one_cumulant_product}
\end{equation}
By (\ref{eq_affine_cumulant}) and the first inequality in
the induction hypothesis,
(\ref{eq_induction_inputs}),
\begin{equation}
 |\kappa_r(t)(v)|_t^2\leq b_r\langle A_tv,v\rangle
 \qquad(2\leq r<m,\ v\in\RR^n).
 \label{eq_posterior_induction}
\end{equation}
This holds almost-surely for $\mu_t$, since its centered push-forward 
under $A_t^{-1/2}$ is again compactly supported, isotropic and
log-concave.

\medskip
Fix $t$ and write $w_j=A_t^{-1/2}e_j$. In the sum
$$
 |P_a(t)|_t^2
 =\sum_{i_2,\ldots,i_m=1}^n
       P_a(t)(w_{i_2},\ldots,w_{i_m})^2,
$$
let us first keep $i_2,\ldots,i_a$ fixed and sum over $i_{a+1},\ldots,i_m$. We let $z\in\RR^n$ be the vector
$$
 z =\kappa_{a+1}(t)(u,w_{i_2},\ldots,w_{i_a},\cdot)
$$
and 
$$
 y =\kappa_{m-a+1}(t)(w_{i_{a+1}},\ldots,w_{i_m},\cdot).
$$
Then,
\begin{align*} 
 P_a(t)(w_{i_2},\ldots,w_{i_m})
 & = \langle z, A_t^{-1} y \rangle = \langle A_t^{-1} z,  y \rangle 
 \\ & =\kappa_{m-a+1}(t)(w_{i_{a+1}},\ldots,w_{i_m}, A_t^{-1} z) \\ & =\kappa_{m-a+1}(t)(A_t^{-1} z, w_{i_{a+1}},\ldots,w_{i_m}).
 \end{align*}
Summing over $i_{a+1},\ldots,i_m$ and using
(\ref{eq_posterior_induction}) with $v=A_t^{-1}z$, we obtain
\begin{align}
 \sum_{i_{a+1},\ldots,i_m=1}^n
       P_a(t)(w_{i_2},\ldots,w_{i_m})^2
 &=|\kappa_{m-a+1}(t)(A_t^{-1}z)|_t^2\notag\\
 &\leq b_{m-a+1}\langle A_t^{-1}z,z\rangle\notag\\
 &=b_{m-a+1}\sum_{k=1}^n
       \kappa_{a+1}(t)(u,w_{i_2},\ldots,w_{i_a},w_k)^2.
 \label{eq_partial_product_sum}
\end{align}
The last equality uses $\sum_kw_k \otimes w_k =A_t^{-1}$, so that
$\langle A_t^{-1}z,z\rangle=\sum_k\langle z,w_k\rangle^2$.
Now sum over $i_2,\ldots,i_a$. Together with the sum over $k$, these indices
range over all inputs of $\kappa_{a+1}(t)(u)$. Hence
\begin{equation}
 |P_a(t)|_t^2\leq b_{m-a+1} \cdot  \cE_{a+1}(t).
 \label{eq_product_energy_bound}
\end{equation}
The second inequality in the induction hypothesis (\ref{eq_induction_inputs}), for  $r = a \leq m-1$
gives
\begin{equation}
 \int_0^\infty\EE|P_a(t)|_t^2\,dt
 \leq b_{m-a+1}\int_0^\infty\EE\cE_{a+1}(t)\,dt
 \leq2b_a b_{m-a+1}.
 \label{eq_product_integral_bound}
\end{equation}

\medskip
For each cardinality $a$, all the $\binom{m-1}{a-1}$ subsets $J$ of size $a$
containing $1$ yield tensors obtained from $P_a(t)$ by permuting its
indices. Therefore their $|\cdot|_t$-norms are equal. The triangle inequality thus gives
\begin{align}
 \left(\int_0^\infty\EE|L_m(t)(u)|_t^2\,dt\right)^{1/2}
 &\leq\sum_{a=2}^{m-2}\binom{m-1}{a-1}
       \left(\int_0^\infty\EE|P_a(t)|_t^2\,dt\right)^{1/2}\notag\\
 &\leq\sqrt2\sum_{a=2}^{m-2}\binom{m-1}{a-1}
                       \sqrt{b_a b_{m-a+1}}.
 \label{eq_lower_cumulant_sum}
\end{align}
By (\ref{eq_induction_constants}), every summand has the same value:
\begin{equation}
 \binom{m-1}{a-1}\sqrt{b_a b_{m-a+1}}
 =\frac{(m-1)!}{(a-1)!(m-a)!}
      K^{(m-1)/2}(a-1)!(m-a)!
 =\sqrt{b_m}.
 \label{eq_factorial_cancellation}
\end{equation}
There are exactly $m-3$ summands, so
\begin{equation}
 \left(\int_0^\infty\EE|L_m(t)(u)|_t^2\,dt\right)^{1/2}
 \leq\sqrt2\,(m-3)\sqrt{b_m}.
 \label{eq_lower_cumulant_integral}
\end{equation}
(In particular, if $m=3$ then $L_3 = 0$.)

\medskip
By Lemma~\ref{lem_integrated_energy}, we conclude that
\begin{align*}
 \cE_m(0)+\frac12\int_0^\infty\EE\cE_{m+1}(t)\,dt
 &\leq2Cm^2b_{m-1}+4(m-3)\sqrt{b_{m-1}b_m}\\
 &=b_m\left(\frac{2C}{K}\frac{m^2}{(m-1)^2}
             +\frac4{\sqrt K}\frac{m-3}{m-1}\right).
\end{align*}
 Taking
$K\geq\max\{5,9C,64\}$ proves (\ref{eq_cumulant_induction}). 
This completes the induction, and thus (\ref{eq_cumulant_induction}) holds for all $m$.
Since
$\cE_m(0)=\langle R_m^\mu u,u\rangle$, we obtain
(\ref{eq_cumulant_bound}) for all compactly supported, isotropic, log-concave probability measures in $\RR^n$.
The bound for non-compactly-supported measures follows by a standard approximation argument. 
\end{proof}

\section{Suspension and the proof of Theorem~\ref{thm_main}}
\label{sec_suspension}

We now pass from the cumulant bound to a bound for the Taylor tensors $\cT_d f$
defined in (\ref{eq_1337}). The idea is to use the
cumulant estimate in a larger dimension, by adding a coordinate
which contains a sum of independent copies of $f$.

\begin{proposition}
\label{prop_suspension}
Fix an integer $d\geq2$. Suppose that $b>0$ satisfies
$R_{d+1}^\nu\leq b\,\id$ for every isotropic log-concave probability
measure $\nu$, in every dimension. Then every isotropic regular
log-concave probability measure $\mu$ satisfies
\begin{equation}
 |\cT_d f|^2\leq\frac{2b}{(d!)^2}\|f\|_{L^2(\mu)}^2
 \qquad \text{for any} \ f\in L^2(\mu).
 \label{eq_suspension_bound}
\end{equation}
\end{proposition}

\begin{proof}
Write $d\mu=e^{-V}dx$ with $\nabla^2V\geq a\,\id$, where $a>0$.
First, assume that  $f \in L^2(\mu)$ is a smooth function with a uniformly bounded hessian, such that
\begin{equation}
 \int_{\RR^n} f\,d\mu=0,\qquad \int_{\RR^n} f^2\,d\mu=1,
 \qquad \int_{\RR^n} xf(x)\,d\mu(x)=0.
 \label{eq_suspension_test}
\end{equation}
Let $X_1,\ldots,X_N$ be independent random vectors with law $\mu$. Independently,
let $\eta_\beta$ be a random variable with density $(\beta/2)e^{-\beta|s|}$ on $\RR$,
where $\beta>0$. Set
$$
 F_N=\frac1{\sqrt N}\sum_{i=1}^N f(X_i),\qquad
 \sigma_\beta^2=1+\frac2{\beta^2},\qquad
 S=\frac{F_N+\eta_\beta}{\sigma_\beta}.
$$
The joint density of $(S,X_1,\ldots,X_N)$ on $\RR^{Nn+1}$ is,
up to normalization, 
\begin{equation}
e^{-W_N(s, x_1,\ldots,x_N)} := \exp \left \{ -\sum_{i=1}^N V(x_i)
 -\beta\left|\sigma_\beta s-\frac1{\sqrt N}
                                    \sum_{i=1}^N f(x_i)\right| \right \},
 \label{eq_suspension_potential}
\end{equation}
where $s \in \RR$ and $x_1,\ldots,x_N \in \RR^n$.
Choose $N$ large enough so that
\begin{equation}
 \frac{\beta}{\sqrt N}\sup_x\|\nabla^2f(x)\|_{op}\leq a.
 \label{eq_suspension_dimension}
\end{equation}
Then $W_N$  is convex. Indeed, it is the
maximum of the two functions
$$
 \sum_{i=1}^N\left(V(x_i)\pm\frac\beta{\sqrt N}f(x_i)\right)
 \ \mp\ \beta\sigma_\beta s,
$$
and each is convex in view of  (\ref{eq_suspension_dimension}).
The joint law of $(S,X_1,\ldots,X_N)$ is  isotropic, since  (\ref{eq_suspension_test}) gives
$$
 \EE S=0,\qquad \EE S^2=1,\qquad
 \EE[SX_i]=\frac{\EE[f(X_i)X_i]}{\sigma_\beta\sqrt N}=0.
$$
Consider the logarithmic Laplace transform 
\begin{align*}  \Lambda(w,z_1,\ldots,z_N)  & = \log \EE \exp \left( wS + \sum_{i=1}^N \langle z_i, X_i \rangle \right) 
\\ & = \log \EE \exp \left( w \frac{\eta_\beta}{\sigma_{\beta}}  \right)  + \log \EE \exp \left(   \frac{w}{\sigma_{\beta} \sqrt{N}} \sum_{i=1}^N f(X_i) + \sum_{i=1}^N \langle z_i, X_i \rangle \right),
\end{align*} 
  where $w\in\RR$ and $z_i\in\RR^n$. The logarithmic Laplace transform 
is finite and smooth in a neighborhood of the origin.  Differentiating in $w$, at $w=0$, gives
\begin{equation}
 \frac{\partial\Lambda}{\partial w}(0,z_1,\ldots,z_N)
 =\frac1{\sigma_\beta\sqrt N}
       \sum_{i=1}^N
       \frac{\int f(x)e^{\langle z_i,x\rangle}\,d\mu(x)}
            {\int e^{\langle z_i,x\rangle}\,d\mu(x)}
 =\frac1{\sigma_\beta\sqrt N}\sum_{i=1}^NF_f(z_i).
 \label{eq_suspension_laplace}
\end{equation}
Let $\nu$ be the probability measure on $\RR^{Nn+1}$ which is the law of $(S, X_1,\ldots,X_N)$, and write
$$
 \kappa=\kappa_{d+1}^\nu=\nabla^{d+1}\Lambda(0).
$$
We index the coordinates of $\RR^{Nn+1}$ by $0$, corresponding
to $S$, and by pairs $(i,j)$, corresponding to the $j$th
coordinate of $X_i$, where $1\leq i\leq N$ and $1\leq j\leq n$.
Write $e_0$ and $e_{i,j}$ for the corresponding standard
basis vectors, and $z_{i,j}$ for the $j$th coordinate of $z_i$.
Differentiating (\ref{eq_suspension_laplace}) gives
\begin{equation}
 \begin{split}
 \kappa_{0,(i_1,j_1),\ldots,(i_d,j_d)}
 &=\frac{\partial^{d+1}\Lambda}
         {\partial w\,\partial z_{i_1,j_1}\cdots
                      \partial z_{i_d,j_d}}(0)\\
 &=
 \begin{cases}
 \displaystyle
 \frac{d!}{\sigma_\beta\sqrt N}
       (\cT_d f)_{j_1,\ldots,j_d},
       & i_1=\cdots=i_d,\\[6pt]
 0,    & \text{otherwise}.
 \end{cases}
 \end{split}
 \label{eq_suspension_cumulant}
\end{equation}
Consequently, the order-$d$ tensor
$\kappa(e_0)=\kappa(e_0,\cdot,\ldots,\cdot)$ decomposes as
$$
 \kappa(e_0)
 =\frac{d!}{\sigma_\beta\sqrt N}
   \sum_{i=1}^N\sum_{j_1,\ldots,j_d=1}^n
   (\cT_d f)_{j_1,\ldots,j_d}\,
   e_{i,j_1}\otimes\cdots\otimes e_{i,j_d}
   +U,
$$
where $U$ is supported on entries for which at least one
of the remaining $d$ indices equals $0$. Therefore
$$
 |\kappa(e_0)|^2
 =\frac{(d!)^2}{\sigma_\beta^2N}
      \sum_{i=1}^N|\cT_d f|^2+|U|^2
 =\frac{(d!)^2}{\sigma_\beta^2}|\cT_d f|^2+|U|^2.
$$
Let us now apply the cumulant bound from our hypothesis for the isotropic, log-concave probability measure $\nu$ in $\RR^{nN+1}$, and the unit vector 
$e_0$. We obtain
$$
 b\geq\langle R_{d+1}^\nu e_0,e_0\rangle
   =|\kappa(e_0)|^2
   \geq\frac{(d!)^2}{\sigma_\beta^2}|\cT_d f|^2.
$$
For each $\beta > 0$ we may choose $N$ satisfying
(\ref{eq_suspension_dimension}). Hence, for any $\beta > 0$,
$$ |\cT_d f|\leq \frac{\sigma_{\beta} \sqrt b}{d!} = \frac{\sqrt{1 + 2/\beta^2}  \cdot \sqrt b}{d!}.  $$
Letting $\beta$ tend to infinity we obtain the desired estimate
\begin{equation}
 |\cT_d f|\leq\frac{\sqrt b}{d!} \qquad \qquad \qquad (d \geq 1).
 \label{eq_suspension_orthogonal}
\end{equation}

\medskip
By a standard approximation argument, the functions satisfying (\ref{eq_suspension_test}) and the stated
smoothness conditions are clearly dense in $L^2(\mu)$-sense in the unit sphere of the orthogonal complement 
to the subspace of affine functions on $\RR^n$. 
Therefore 
(\ref{eq_suspension_orthogonal}) holds for every $f \in L^2(\mu)$ of norm one, 
that is orthogonal to all affine functions.

\medskip Next, for a linear function $\ell(x)=\langle v,x\rangle$, isotropicity gives
$\|\ell\|_2=|v|$, and
$$
 F_\ell(z)=\langle v,\nabla\Lambda_\mu(z)\rangle,
 \qquad
 \cT_d\ell=\frac1{d!}\kappa_{d+1}^\mu(v),
$$
where we recall that $\kappa_{d+1}^\mu(v) = \kappa_{d+1}^\mu(v,\cdot,\ldots,\cdot)$.
Hence $|\cT_d\ell|\leq\sqrt b\,\|\ell\|_2/d!$ by our 
cumulant hypothesis. Decompose an arbitrary $f\in L^2(\mu)$ as
$f=c+\ell+h$, with the three terms mutually orthogonal, where $c$
is constant, $\ell$ is linear, and $h$ is orthogonal to affine 
functions. Since $\cT_d c=0$,
$$
 |\cT_d f|
 \leq\frac{\sqrt b}{d!}(\|\ell\|_2+\|h\|_2)
  \leq\frac{\sqrt{2b}}{d!} \sqrt{\|\ell\|_2^2 + \| h \|_2^2} \leq 
 \frac{\sqrt{2b}}{d!}\|f\|_2.
$$
This completes the proof of  (\ref{eq_suspension_bound}).
\end{proof}

\begin{proof}[Proof of Theorem~\ref{thm_main}]
Let $\mu$ be an isotropic, log-concave probability measure on $\RR^n$. 
As indicated in Section~\ref{sec_background},
by a standard approximation argument, when proving
Theorem~\ref{thm_main} we may assume that $\mu$ is regular. 
Theorem~\ref{thm_cumulants}
and Proposition~\ref{prop_suspension}, with
$b=K^d(d!)^2$, give
$$
 |\cT_d f|^2\leq2K^d\|f\|_2^2\leq(2K)^d\|f\|_2^2
 \qquad(d\geq2).
$$
For $d=1$, isotropy gives
$|\cT_1f|=|\int xf(x)\,d\mu(x)|\leq\|f\|_2$.
Thus the hypothesis of Theorem~\ref{thm_1032} holds with
$R=\sqrt{2K}$, and $C_P(\mu)\leq2CK$ is universally bounded.
\end{proof}

\bigskip
\noindent Department of Mathematics, Weizmann Institute of Science, Rehovot, Israel. \\
\textit{e-mail:} \texttt{pierre.bizeul@weizmann.ac.il}

\bigskip
\noindent School of Mathematical Sciences, Tel Aviv University, Tel Aviv, Israel; and \\
Department of Mathematics, Weizmann Institute of Science, Rehovot, Israel. \\
\textit{e-mail:} \texttt{klartagb@tau.ac.il}

\bigskip
\noindent Institut de Mathématiques de Jussieu--Paris Rive Gauche,
Équipe d'Analyse Fonctionnelle, Sorbonne Université, Paris, France. \\
\textit{e-mail:} \texttt{joseph.lehec@imj-prg.fr}

\end{document}